\documentclass[11pt,reqno]{amsart}
\usepackage{fullpage,times,graphicx,amssymb,amsmath,psfrag,xcolor,natbib}
\usepackage[colorlinks,citecolor=bbluegray,linkcolor=ddarkbrown,urlcolor=blue,breaklinks]{hyperref}
\usepackage{bbding}
\usepackage{algorithmic,algorithm}

\usepackage[off]{auto-pst-pdf} 
\usepackage{comment}

\definecolor{ddarkbrown}{rgb}{0.5,0.2,0.05} \definecolor{bbluegray}{rgb}{0.05,0,0.5}

\newtheorem{theorem}{Theorem}[section]
\newtheorem{proposition}[theorem]{Proposition}
\newtheorem{definition}[theorem]{Definition}
\newtheorem{lemma}[theorem]{Lemma}
\newtheorem{corollary}[theorem]{Corollary}
\renewenvironment{proof}{\textbf{Proof.}}{\QED\bigskip}
\newtheorem{assumption}[theorem]{Assumption}

\def\xx{{\boldsymbol x}}
\def\yy{{\boldsymbol y}}
\def\vv{{\boldsymbol v}}

\newcommand{\BEAS}{\begin{eqnarray*}}
\newcommand{\EEAS}{\end{eqnarray*}}
\newcommand{\BEA}{\begin{eqnarray}}
\newcommand{\EEA}{\end{eqnarray}}
\newcommand{\BEQ}{\begin{equation}}
\newcommand{\EEQ}{\end{equation}}
\newcommand{\BIT}{\begin{itemize}}
\newcommand{\EIT}{\end{itemize}}
\newcommand{\BNUM}{\begin{enumerate}}
\newcommand{\ENUM}{\end{enumerate}}

\newcommand{\BA}{\begin{array}}
\newcommand{\EA}{\end{array}}

\newcommand{\ones}{\mathbf 1}

\newcommand{\reals}{{\mathbb R}}

\newcommand{\diam}{\mathop{\bf diam}}

\newcommand{\Rank}{\mathop{\bf Rank}}

\newcommand{\Card}{\mathop{\bf Card}}

\newcommand{\Prob}{\mathop{\bf Prob}}

\newcommand{\Co}{{\mathop {\bf Co}}}

\newcommand{\Po}{{\mathop {\bf Po}}}
\newcommand{\Ext}{{\mathop {\bf Ext}}}

\newcommand{\dist}{\mathop{\bf dist{}}}

\newcommand{\QED}{~~\rule[-1pt]{6pt}{6pt}}

\newcommand{\argmin}{\mathop{\rm argmin}}
\newcommand{\epi}{\mathop{\bf epi}}

\newcommand{\dom}{\mathop{\bf dom}}

\let \oldsection \section
\renewcommand{\section}{\vspace{3ex plus 1ex}\oldsection}

\begin{document}
\title{An Approximate Shapley-Folkman Theorem.}

\author{Thomas Kerdreux}
\address{D.I., UMR 8548,\vskip 0ex
\'Ecole Normale Sup\'erieure, Paris, France.}
\email{thomaskerdreux@gmail.com}

\author{Igor Colin}
\address{Huawei, France.}
\email{igor.colin@gmail.com}

\author{Alexandre d'Aspremont}
\address{CNRS \& D.I., UMR 8548,\vskip 0ex
\'Ecole Normale Sup\'erieure, Paris, France.}
\email{aspremon@ens.fr}

\date{\today}
\subjclass[2010]{}

\begin{abstract}
The Shapley-Folkman theorem shows that Minkowski averages of uniformly bounded sets tend to be convex when the number of terms in the sum becomes much larger than the ambient dimension. In optimization, \citet{Aubi76} show that this produces an a priori bound on the duality gap of separable nonconvex optimization problems involving finite sums. This bound is highly conservative and depends on unstable quantities, and we relax it in several directions to show that non convexity can have a much milder impact on finite sum minimization problems such as empirical risk minimization and multi-task classification. As a byproduct, we show a new version of Maurey's classical approximate Carath\'eodory lemma where we sample a significant fraction of the coefficients, without replacement, as well as a result on sampling constraints using an approximate Helly theorem, both of independent interest.
\end{abstract}
\maketitle

\section{Introduction}\label{s:intro}
We focus on separable optimization problems written
\BEQ\label{eq:ncvx-pb}\tag{P}
\BA{ll}
\mbox{minimize} & \sum_{i=1}^{n} f_i(x_i) \\
\mbox{subject to} & Ax\leq b,\\
& x_i \in Y_i, \quad i=1,\ldots,n,
\EA\EEQ
in the variables $x_i\in\reals^{d_i}$ with $d=\sum_{i=1}^n d_i$, where the functions $f_i$ are lower semicontinuous (but not necessarily convex), the sets $Y_i \subset \dom f_i$ are compact, and $A\in\reals^{m \times d}$ and $b\in\reals^m$. \citet{Aubi76} showed that the duality gap of problem~\eqref{eq:ncvx-pb} vanishes when the number of terms $n$ grows towards infinity while the dimension $m$ remains bounded, provided the nonconvexity of the functions $f_i$ is uniformly bounded. The result in \citep{Aubi76} hinges on the fact that the epigraph of problem ~\eqref{eq:ncvx-pb} can be written as a Minkowski sum of $n$ sets in dimension $m$. In this setting, the Shapley-Folkman theorem shows that if $V_i \subset\reals^m$, $i=1,\ldots,n$ are arbitrary subsets of $\reals^m$ and
\[
x\in\Co\left(\sum_{i=1}^n V_i\right)
\qquad
\mbox{then}
\qquad
x \in~ \sum_{[1,n]\setminus \mathcal{S}} V_i ~ + ~  \sum_{\mathcal{S}} \Co(V_i)
\]
for some $|\mathcal{S}|\leq m$. If the sets $V_i$ are uniformly bounded, $n$ grows and $m$ remains bounded, the term $\sum_{\mathcal{S}} \Co(V_i)$ becomes negligible and the Minkowski sum $\sum_i V_i$ is increasingly close to its convex hull. In fact, several measures of nonconvexity decrease monotonically towards zero when $n$ grows in this setting, with \citep{Frad17} showing for instance that the Hausdorff distance 
\[
d_H\left(\sum_i V_i,\Co\left(\sum_i V_i\right)\right) \rightarrow 0.\] 
We illustrate this phenomenon graphically in Figure~\ref{fig:l12ball}, where we show the Minkowski mean of $n$ unit~$\ell_{1/2}$ balls for $n=1,2,10,\infty$ in dimension 2, and the average of five arbitrary point sets (defined from digits here). In both cases, Minkowski averages are nearly convex for relatively small values of $n$.

The Shapley-Folkman theorem was derived by Shapley \& Folkman in private communications and first published by \citep{Star69}. It was used by \citet{Aubi76} to derive a priori bounds on the duality gap. The continuous limit of this result is known as the Liapunov convexity theorem and shows that the range of non-atomic, vector valued measures is convex \citep{Auma65,Berl73}. The results of \citet{Aubi76} were extended in \citep{Ekel99} to generic separable constrained problems, and also by \citep{Laue82,Bert14} to more precise yet less explicit nonconvexity measures, who describe applications to large-scale unit commitment problems. Extreme points of the set of solutions of a convex relaxation to problem~\eqref{eq:ncvx-pb} are used to produce good approximations and \citep{Udel16} describe a randomized purification procedure to find such points with probability one.

The Shapley-Folkman theorem is a direct consequence of the conic version of Carath\'eodory's theorem, with the number of terms in the conic representation of optimal points controlling the duality gap bound. Our first contribution seeks to reduce this number by allowing a small approximation error in the conic representation. This essentially trades off approximation error with duality gap. In general, these approximations are handled by Maurey's classical approximate Carath\'eodory lemma \citep{Pisi81}. Here however we need to sample a very high fraction of the coefficients, hence we produce a {\em high sampling ratio} version of the approximate Carath\'eodory lemma using results by \citep{Serf74,Bard15,Schn16} on sampling sums without replacement. 

The duality gap bounds produced using the Shapley-Folkman theorem also directly depend on the number of active constraints at the optimum. Reducing this number using different representations or approximations of the feasible set thus has a significant impact on the tightness of these bounds. In this vein, we show a constraint sampling result, of independent interest, using recent results by \citep{Adip19} on an approximate version of Helly's theorem.

We then use these results to produce an approximate version of the duality gap bound in \citep{Aubi76} which allows a direct tradeoff between the impact of nonconvexity and the approximation error. This approximate formulation also has the benefit of writing the gap bound in terms of stable quantities, thus better revealing the link between problem structure and duality gap.


Nonconvex separable problems involving finite sums such as~\eqref{eq:ncvx-pb} occur naturally in machine learning, signal processing and statistics. The most direct examples being perhaps empirical risk minimization and multi-task learning. In this later setting, our bounds show that when the number of tasks grows and the tasks are only loosely coupled (e.g. the separable $\ell_2$ constraint \citep{Cili17}), nonconvex multi-task problems have asymptotically vanishing duality gap. A stream of recent results have shown that finite sum optimization problems have particularly good computational complexity (see \citep{Roux12,John13,Defa14} and more recently \citep{Alle16,Redd16} in the nonconvex case), our results show that {\em they also have intrinsically low duality gap in some settings.}



\begin{figure}[h]
\begin{center}
\includegraphics[width=0.7\textwidth]{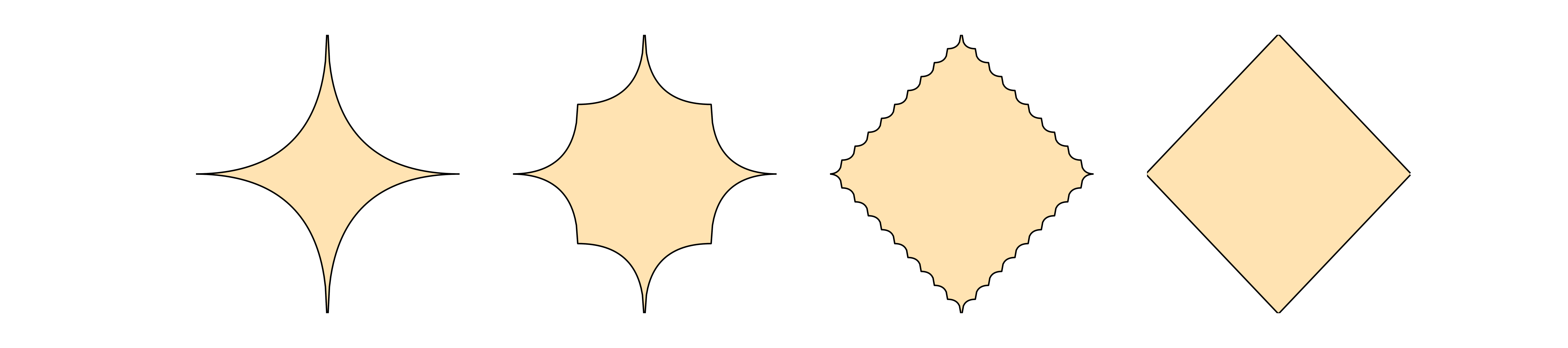}
\vskip 3ex
\begin{tabular}{ccccccccccc}
\raisebox{-0.4\height}{\includegraphics[scale=0.05]{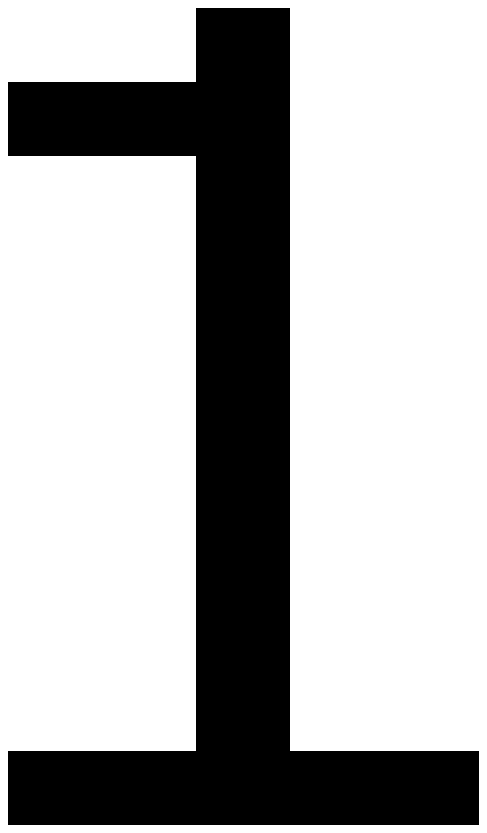}} &
+&
\raisebox{-0.4\height}{\includegraphics[scale=0.05]{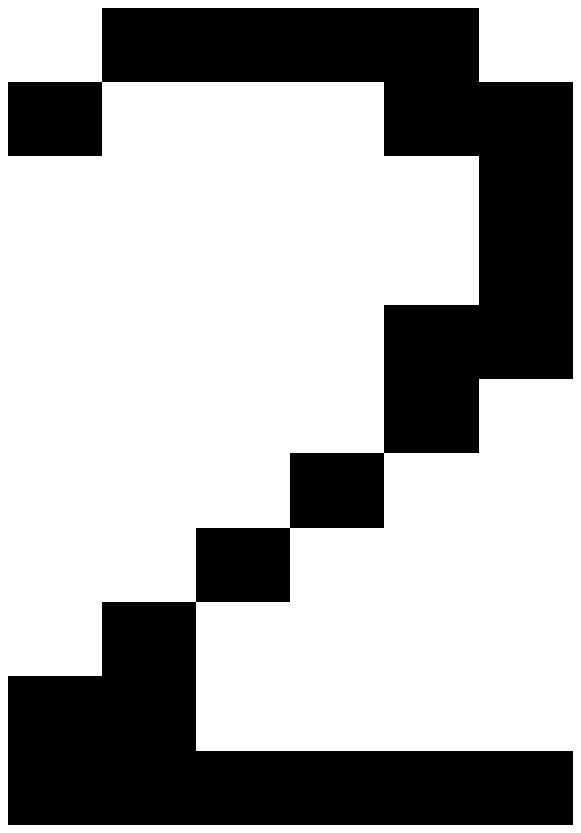}} &
+&
\raisebox{-0.4\height}{\includegraphics[scale=0.05]{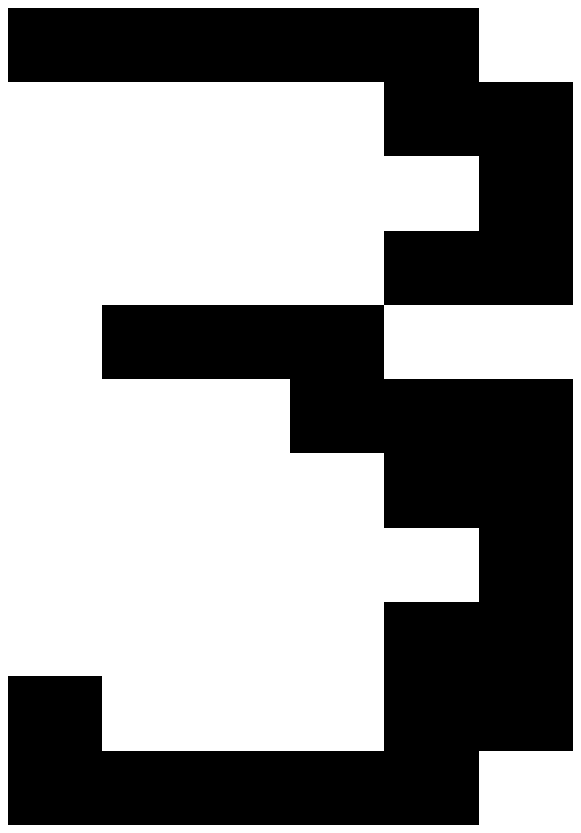}} &
+&
\raisebox{-0.4\height}{\includegraphics[scale=0.05]{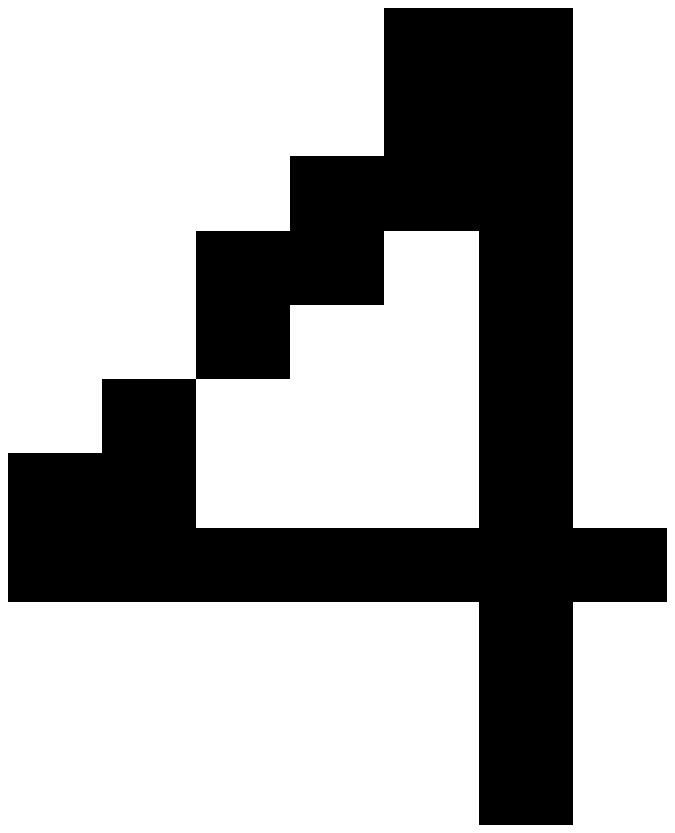}} &
+&
\raisebox{-0.4\height}{\includegraphics[scale=0.05]{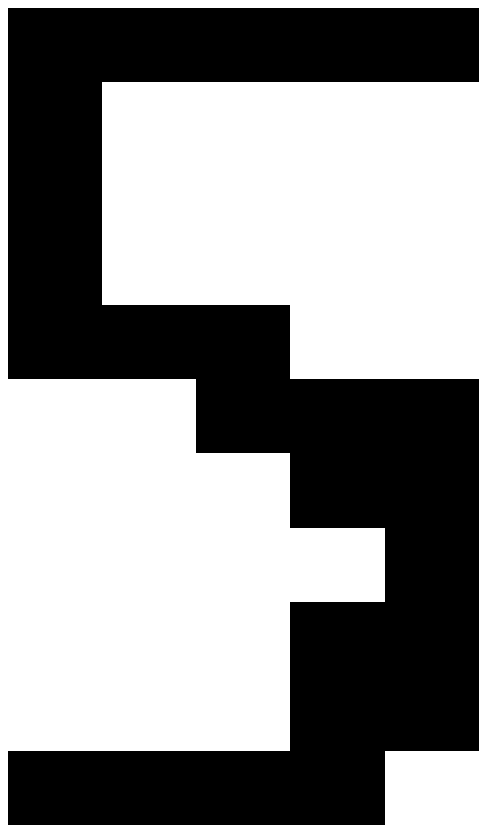}} &
=&
$5 \times$ \raisebox{-0.4\height}{\includegraphics[scale=0.07]{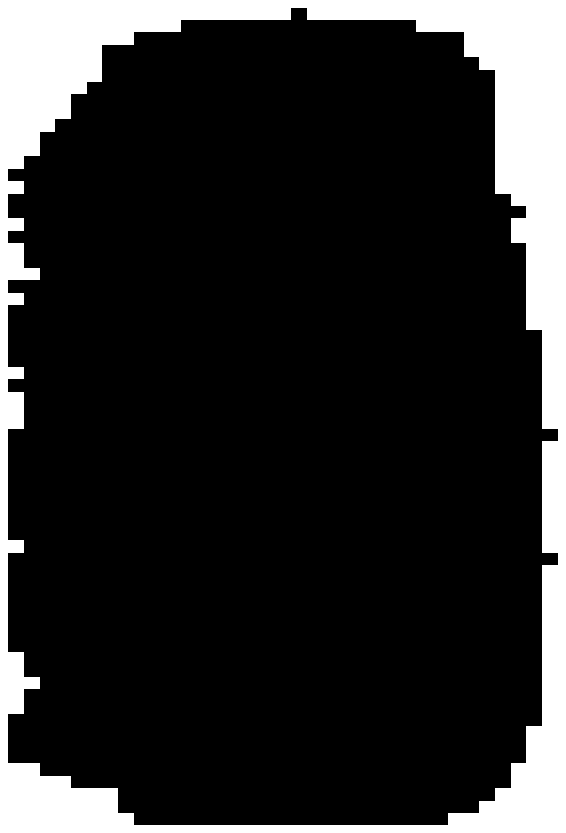}}
\end{tabular}
\end{center}
\caption{{\em Top:} The $\ell_{1/2}$ ball, Minkowsi average of two and ten balls, and convex hull. {\em Bottom:} Minkowsi average of five first digits (obtained by sampling). \label{fig:l12ball}}
\end{figure}

\section{Convex Relaxation and Bounds on the Duality Gap}\label{s:relax}
We first recall and adapt some key results from \citep{Aubi76,Ekel99} producing {\em a priori} bounds on the duality gap, using an epigraph formulation of problem~\eqref{eq:ncvx-pb}.

\subsection{Biconjugate and Convex Envelope}
Assuming that $f$ is not identically $+\infty$ and is minorized by an affine function, we write $f^*(y)\triangleq \inf_{x\in\dom f} \{y^{\top}x - f(x)\}$ the conjugate of $f$, and $f^{**}(y)$ its biconjugate. The biconjugate of $f$ (aka the convex envelope of $f$) is the pointwise supremum of all affine functions majorized by $f$ (see e.g. \citep[Th.\,12.1]{Rock70} or \citep[Th.\,X.1.3.5]{Hiri93}), a corollary then shows that $\epi(f^{**})=\overline{\Co(\epi(f))}$. For simplicity, we write $S^{**}=\overline{\Co(S)}$ for any set $S$ in what follows. We will make the following technical assumptions on the functions $f_i$.
\begin{assumption}\label{as:fi}
The functions $f_i: \reals^{d_i} \rightarrow \reals$ are proper, 1-coercive, lower semicontinuous and there exists an affine function minorizing them.
\end{assumption}
Note that coercivity trivially holds if $\dom(f_i)$ is compact (since $f$ is $+\infty$ outside). When Assumption~\ref{as:fi} holds, $\epi(f^{**})$, $f_i^{**}$ and hence $\sum_{i=1}^{n} f_i^{**}(x_i)$ are closed \citep[Lem.\,X.1.5.3]{Hiri93}. Finally, as in e.g. \citep{Ekel99}, we define the lack of convexity of a function as follows.

\begin{definition}\label{def:rho}
Let $f: \reals^{d} \rightarrow \reals$ we let $\rho(f)\triangleq \sup_{x\in \dom(f)} \{f(x) - f^{**}(x)\}$.
\end{definition}

Many other quantities measure lack of convexity, see e.g. \citep{Aubi76,Bert14} for further examples. In particular,  the nonconvexity measure $\rho(f)$ can be further refined, using the fact that 
\[
\rho(f)=\sup_{\substack{x_i\in \dom(f)\\ \alpha\in\reals^{d+1}}}~\left\{ f\left(\sum_{i=1}^{d+1}\alpha_i x_i\right) -  \sum_{i=1}^{d+1}\alpha_i f(x_i): \ones^T\alpha=1,\alpha \geq 0\right\}
\]
when $f$ satisfies Assumption~\ref{as:fi} (see \citep[Th.\,X.1.5.4]{Hiri93}). In this setting, \citet{Bi16} define the $k^{th}$-nonconvexity measure as
\BEQ\label{def:rhok}
\rho_k(f)\triangleq\sup_{\substack{x_i\in \dom(f)\\ \alpha\in\reals^{d+1}}}~\left\{ f\left(\sum_{i=1}^{d+1}\alpha_i x_i\right) -  \sum_{i=1}^{d+1}\alpha_i f(x_i): \ones^T\alpha=1,\Card(\alpha)\leq k,\alpha \geq 0\right\}
\EEQ
which restricts the number of nonzero coefficients in the formulation of $\rho(f)$. Note that $\rho_1(f)=0$.

\subsection{Convex Relaxation}
We will now show that the dual of problem~\eqref{eq:ncvx-pb} maximizes a linear form over the convex hull of a Minkowski sum of $n$ epigraphs. We also show that this dual matches the dual of a convex relaxation of~\eqref{eq:ncvx-pb}, formed using the convex envelopes of the functions $f_i(x)$. In what follows, we will assume without loss of generality that $Y_i=\reals^{d_i}$, replacing $f_i$ by $f_i(x)+\mathbf{1}_{Y_i}(x)$. We use the biconjugate to produce a convex relaxation of problem~\eqref{eq:ncvx-pb} written
\BEQ\label{eq:cvx-pb}\tag{CoP}
\BA{ll}
\mbox{minimize} & \sum_{i=1}^{n} f_i^{**}(x_i) \\
\mbox{subject to} & Ax\leq b\\
\EA\EEQ
in the variables $x_i\in\reals^{d_i}$. Writing the epigraph of problem~\eqref{eq:ncvx-pb} as in \citep[\S5.3]{Boyd03} or \citep{Lema01},
\[
\mathcal{G}\triangleq \left\{(x,r_0,r) \in \reals^{d+1+m} :~ \sum_{i=1}^{n} f_i(x_i) \leq r_0,\,Ax-b \leq r \right\},
\]
and its projection on the last $m+1$ coordinates,
\BEQ\label{eq:epi}
\mathcal{G}_r\triangleq \left\{(r_0,r) \in \reals^{m+1} :~  (x,r_0,r)\in \mathcal{G}\right\},
\EEQ
we can write the Lagrange dual function of~\eqref{eq:ncvx-pb} as 
\BEQ\label{eq:dual-f}
\Psi(\lambda) \triangleq \inf\left\{ r_0 + \lambda^{\top}r:~ (r_0,r) \in \mathcal{G}_r^{**} \right\},
\EEQ
in the variable $\lambda\in\reals^m$, where $\mathcal{G}^{**}=\overline{\Co(\mathcal{G})}$ is the closed convex hull of the epigraph $\mathcal{G}$ (the projection being linear here, we have $(\mathcal{G}_r)^{**}=(\mathcal{G}^{**})_r=\mathcal{G}_r^{**}$). We need constraint qualification conditions for strong duality to hold in~\eqref{eq:cvx-pb} and we now recall the result in \citep[Th.\,2.11]{Lema01} which shows that because the explicit constraints are affine here, the dual functions of~\eqref{eq:ncvx-pb} and~\eqref{eq:cvx-pb} are equal. The (common) dual of~\eqref{eq:ncvx-pb} and~\eqref{eq:cvx-pb} is then
\BEQ\label{eq:d-pb}\tag{D}
\sup_{\lambda \geq 0} \Psi(\lambda)
\EEQ
in the variable $\lambda\in\reals^m$.  The following result shows that strong duality holds under mild technical assumptions.

\begin{theorem}{\citep[Th.\,2.11]{Lema01}}
The function $\Psi(\lambda)$ is also the dual function associated with~\eqref{eq:cvx-pb}. Assuming that $\Psi$ is not constant equal to $-\infty$ and that there is a feasible $x$ in the relative interior of $\dom\left(\sum_{i=1}^{n} f_i^{**}\right)$ then $\Psi$ attains its maximum and 
\[
\max_{\lambda} \Psi(\lambda) = \inf\left\{ \sum_{i=1}^{n} f_i^{**}(x_i):~ x\in\reals^d,\,Ax\leq b\right\}
\]
i.e. strong duality holds.
\end{theorem}
This last result shows that the convex problem~\eqref{eq:cvx-pb} indeed shares the same dual as problem~\eqref{eq:ncvx-pb}.

\subsection{Perturbations}\label{ss:pert}
In the next section, perturbed versions of problems~\eqref{eq:ncvx-pb} and~\eqref{eq:cvx-pb} will emerge to quantify our approximation bounds. These are written respectively
\BEQ\label{eq:p-ncvx-pb}\tag{pP}
\BA{rll}
\mathrm{h}_{P}(u) \triangleq &\mbox{min.} & \sum_{i=1}^{n} f_i(x_i) \\
&\mbox{s.t.} & Ax - b \leq u\\
&& x_i \in Y_i, \quad i=1,\ldots,n,
\EA\EEQ
in the variables $x_i\in\reals^{d_i}$, with perturbation parameter $u\in\reals^m$, and
\BEQ\label{eq:p-cvx-pb}\tag{pCoP}
\BA{rll}
\mathrm{h}_{CoP}(u) \triangleq &\mbox{min.} & \sum_{i=1}^{n} f_i^{**}(x_i) \\
&\mbox{s.t.} & Ax - b \leq u\\
\EA\EEQ
in the variables $x_i\in\reals^{d_i}$, with perturbation parameter $u\in\reals^m$.

\subsection{Bounds on the Duality Gap}\label{ss:sf}
We now recall results by \citep{Aubi76,Ekel99} bounding the duality gap in~\eqref{eq:ncvx-pb} using the lack of convexity of the functions $f_i$. In the formulation below, the dual is more explicit than in \citep{Ekel99} because the constraints are affine here.

\begin{proposition}\label{prop:gap}
Suppose the functions $f_i$ in~\eqref{eq:ncvx-pb} satisfy Assumption~\ref{as:fi}. There is a point $x^\star\in\reals^d$ at which the primal optimal value of~\eqref{eq:cvx-pb} is attained, such that
\BEQ\label{eq:gap-bnds}
\underbrace{\sum_{i=1}^{n} f_i^{**}(x^\star_i)}_{\ref{eq:cvx-pb}} ~\leq~ \underbrace{\sum_{i=1}^{n} f_i(\hat x^\star_i)}_{\ref{eq:ncvx-pb}} ~\leq~ \underbrace{\sum_{i=1}^{n} f_i^{**}(x^\star_i)}_{\ref{eq:cvx-pb}} ~+~ \underbrace{\sum_{i \in \mathcal{S}} \rho(f_i)}_\mathrm{gap}
\EEQ
with $\hat x^\star$ is an optimal point of~\eqref{eq:ncvx-pb}, and
\[
\mathcal{S}\triangleq \left\{ i:~ \left(f_i^{**}(x_i^\star), A_ix_i^\star\right) \notin \Ext(\mathcal{F}_i)\right\}
\]
where $\mathcal{F}_i\subset \reals^{m+1}$ is defined as
\[
\mathcal{F}_i = \left\{ \left(f_i^{**}(x_i), A_ix_i\right):~ x_i \in \reals^{d_i} \right\}
\]
writing $A_i \in \reals^{m \times d_i}$ the $i^{th}$ block of $A$.
\end{proposition}
\begin{proof}
Using \citep[Cor.\,A.6]{Lema01}, we know
\[
\mathcal{G}_r^{**} = \left\{(r_0, r) \in \reals^{m+1}:~ \sum_{i=1}^{n} f_i^{**}(x_i)\leq r_0,\, Ax-b \leq r\right\}.
\]
Since $\mathcal{G}_r^{**}$ is closed by construction and the sets $\mathcal{F}_i$ are closed by Assumption~\ref{as:fi}, there is a point 
$x^\star\in \mathbb{R}^d$ which attains the primal optimal value in~\eqref{eq:cvx-pb}. We write the corresponding minimizer of~\eqref{eq:dual-f} in $\mathcal{G}_r^{**}$ as
\BEQ\label{eq:zstar}
z^\star = \sum_{i=1}^n
\begin{pmatrix}
f_i^{**}(x^\star_i)\\
A_ix^\star_i
\end{pmatrix}
+
\begin{pmatrix}
0\\
w-b
\end{pmatrix}
\EEQ
with $w \in \reals^{m}_+$, which we summarize as
\[
z^\star = \sum_{i=1}^n z^{(i)} +
\begin{pmatrix}
0\\
w-b
\end{pmatrix},
\]
where $z^{(i)} \in \mathcal{F}_i$. Since $f_i^{**}(x)=f_i(x)$ when $x \in \Ext(\mathcal{F}_i)$ because $\epi(f^{**})=\overline{\Co(\epi(f))}$ when Assumption~\ref{as:fi} holds, we have

\BEAS
\overbrace{\sum_{i=1}^{n} f_i^{**}(x^\star_i)}^{\ref{eq:cvx-pb}} &=& \sum_{i\in [1,n] \setminus \mathcal{S}} f_i^{**}(x^\star_i) + \sum_{i\in \mathcal{S}}f_i(x^\star_i) + \sum_{i\in \mathcal{S}} f_i^{**}(x^\star_i) - f_i(x^\star_i) \\
& \geq &  \sum_{i\in [1,n] \setminus \mathcal{S}} f_i(x^\star_i) + \sum_{i\in \mathcal{S}}f_i(x^\star_i) ~-~ \sum_{i\in \mathcal{S}} \rho(f_i)\\
& \geq & \underbrace{\sum_{i=1}^{n} f_i(\hat{x}^\star_i)}_{\ref{eq:ncvx-pb}}  ~-~ \sum_{i\in \mathcal{S}} \rho(f_i)~.
\EEAS
The last inequality holds because $x^\star$ is feasible for~\eqref{eq:ncvx-pb}.
\end{proof}

This last result bounds {\em a priori} the duality gap in problem~\eqref{eq:ncvx-pb} by 
\[
\sum_{i \in \mathcal{S}} \rho(f_i)
\]
where $\mathcal{S}\subset[1,n]$. The dual problem in~\eqref{eq:d-pb} shows that the optimal solution maximizes an affine form over the closed convex hull of the epigraph of the primal~\eqref{eq:ncvx-pb} and is thus attained at an extreme point of that epigraph. Separability means this epigraph is the Minkowski sum of the closed convex hulls of the epigraphs of the $n$ subproblems, while $|\mathcal{S}|$ counts the number of terms in this sum for which the optimum is attained at an extreme point of these subproblems. The Shapley-Folkman theorem together with the results of the next sections will produce upper bounds on the size of $\mathcal{S}$ and show that it is typically much smaller than $n$.

\section{The Shapley-Folkman Theorem}\label{s:sf}
Carath\'eodory's theorem is the key ingredient in proving the Shapley-Folkman theorem. We begin by recalling Carath\'eodory's result, and its conic formulation, which underpin all the other results in this section. 
\begin{theorem}[Carath\'eodory]
Let $V\subset \reals^n$, then $x\in \Co(V)$ if and only if
\[
x=\sum_{i=1}^{n+1} \lambda_i v_i
\]
for some $v_i\in V$, $\lambda_i\geq 0$ and $\ones^{\top}\lambda=1$.
\end{theorem}

Similarly, if we write $\Po(V)$ the conic hull of $V$, with $\Po(V)=\{\sum_i \lambda_i v_i: v_i \in V, \lambda_i \geq 0$, we have the following result (see e.g. \citep[Cor.\,17.1.2]{Rock70}).

\begin{theorem}[Conic Carath\'eodory]
Let $V\subset \reals^n$, then $x\in \Po(V)$ if and only if
\[
x=\sum_{i=1}^{n} \lambda_i v_i
\]
for some $v_i\in V$, $\lambda_i\geq 0$.
\end{theorem}

The Shapley-Folkman theorem below was derived by Shapley \& Folkman in private communications and first published by \citep{Star69}. 
\begin{theorem}[Shapley-Folkman]\label{th:sf}
Let $V_i \in\reals^d$, $i=1,\ldots,n$ be a family of subsets of $\reals^d$. If
\[
x\in\Co\left(\sum_{i=1}^n V_i\right) = \sum_{i=1}^n\Co\left(V_i\right)
\]
then
\[
x \in~ \sum_{[1,n]\setminus \mathcal{S}} V_i ~ + ~  \sum_{\mathcal{S}} \Co(V_i)
\]
where $|\mathcal{S}|\leq d$.
\end{theorem}
\begin{proof}
Suppose $x\in \sum_{i=1}^n\Co\left(V_i\right)$, then by Carath\'eodory's theorem we can write $x=\sum_{i=1}^n\sum_{j=1}^{d+1} \lambda_{ij} v_{ij}$ where $v_{ij} \in V_i$ and $\lambda_{ij}\geq 0$ with $\sum_{j=1}^{d+1}\lambda_{ij}=1$. These constraints can be summarized as
\BEQ\label{eq:sf-conic}
z=\sum_{i=1}^n\sum_{j=1}^{d+1} \lambda_{ij} z_{ij}
\EEQ
where $z\in\reals^{d + n}$ and
\[
z=
\begin{pmatrix}
x \\ \ones_{n}
\end{pmatrix}
,\quad
z_{ij}=
\begin{pmatrix}
v_{ij} \\ e_i
\end{pmatrix}
,\quad\mbox{for $i=1,\ldots,n$ and $j=1,\ldots,d+1$,}
\]
with $e_i \in \reals^n$ is the Euclidean basis. Since $z$ is a conic combination of the $z_{ij}$, there exist coefficients $\mu_{ij} \geq 0$ such that $z = \sum_{i=1}^n \sum_{j = 1}^{d + 1} \mu_{ij} z_{ij}$ and at most $d + n$ coefficients $\mu_{ij}$ are nonzero. Then, $\sum_{j=1}^{d+1}\mu_{ij}=1$ means that a single $\mu_{ij}=1$ for $i \in [1,n]\setminus \mathcal{S}$ where $|\mathcal{S}| \leq d$ (since $n+d$ nonzero coefficients are spread among $n$ sets, with at least one nonzero coefficient per set), and $\sum_{j=1}^{d+1} \mu_{i} v_{ij} \in V_{i}$ for $i \in [1,n]\setminus \mathcal{S}$.
\end{proof}

This theorem has been used, for example, to prove existence of equilibria in markets with a large number of agents with non-convex preferences. Classical proofs usually rely on a dimension argument \citep{Star69}, but the one we recalled here is more constructive. In what follows, we will use it to bound the duality gap in~\eqref{eq:gap-bnds}.

\section{Stable Bounds on the Duality Gap}\label{s:bnds-gap}
The Shapley-Folkman theorem above was used to produce a priori bounds on the duality gap in \citep{Aubi76}, see also \citep{Ekel99,Bert14,Udel16} for a more recent discussion. We first recall the following result, similar in spirit to that in \citep{Aubi76}.

\begin{proposition}\label{prop:gap-sf}
Suppose the functions $f_i$ in~\eqref{eq:ncvx-pb} satisfy Assumption~\ref{as:fi}. There is a point $x^\star\in\reals^d$ at which the primal optimal value of~\eqref{eq:cvx-pb} is attained, such that
\BEQ\label{eq:sf-bnds}
\underbrace{\sum_{i=1}^{n} f_i^{**}(x^\star_i)}_{\ref{eq:cvx-pb}} ~\leq~ \underbrace{\sum_{i=1}^{n} f_i(\hat x^\star_i)}_{\ref{eq:ncvx-pb}} ~\leq~ 
\underbrace{\sum_{i=1}^{n} f_i^{**}(x^\star_i)}_{\ref{eq:cvx-pb}} ~+~ \underbrace{~\sum_{i=1}^{m+1} \rho(f_{[i]})}_\mathrm{gap}
\EEQ
where $\hat x^\star$ is an optimal point of~\eqref{eq:ncvx-pb} and $\rho(f_{[1]})\geq \rho(f_{[2]}) \geq \ldots \geq \rho(f_{[n]})$.
\end{proposition}
\begin{proof}
Notice that the closed convex hull $\mathcal{G}_r^{**}$ of the epigraph of problem~\eqref{eq:ncvx-pb} can be written as a Minkowski sum, with
\[
\mathcal{G}_r^{**}=\sum_{i=1}^n \mathcal{F}_i + (0,-b) + \reals_+^{m+1},
\quad
\mbox{where}
\quad
\mathcal{F}_i = \left\{ \left(f_i^{**}(x_i), A_ix_i\right):~ x_i \in \reals^{d_i} \right\}\subset\reals^{m+1}
\]
The Krein-Milman theorem shows 
\[
\mathcal{G}_r^{**}=\sum_{i=1}^n \Co\left( \Ext(\mathcal{F}_i)\right) + (0,-b) + \reals_+^{m+1} .
\]
Now, since $\mathcal{F}_i\subset\reals^{m+1}$, the Shapley Folkman Theorem~\ref{th:sf} shows that the point $z^\star\in\mathcal{G}_r^{**}$ in~\eqref{eq:zstar} satisfies
\[
z^\star\in \sum_{[1,n]\setminus \mathcal{S}} \Ext(\mathcal{F}_i) ~ + ~  \sum_{\mathcal{S}} \Co\left( \Ext(\mathcal{F}_i)\right)
\]
for some set $\mathcal{S}\subset[1,n]$ with $|\mathcal{S}|\leq m+1$. This means that we can take $|\mathcal{S}|\leq m+1$ in Proposition~\ref{prop:gap} and yields the desired result.
\end{proof}

The result above directly links the gap bound with the {\em number of nonzero coefficients in the conic combination} defining the solution $z^\star$ in~\eqref{eq:zstar}. The smaller this number, the tighter the gap bound. In fact,
if we use the $k^{th}$-nonconvexity measure $\rho_k(f)$ in~\eqref{def:rhok} instead of $\rho(f)$, the duality gap \citep{Bi16} bound can be refined to
\[
\mathrm{gap} \leq \max_{\beta_i\in[1,m+2]}~\left\{\sum_{i=1}^{n} \rho_{\beta_i}(f_i): \sum_{i=1}^n \beta_i = n+m+1 \right\}.
\]
Because $\rho_1(f)=0$, this last bound can be significantly smaller, since the result in \citep{Aubi76} implicitly assumes that $\sum_{i=1}^n \beta_i = n+(m+2)(m+1)$, instead of $n+m+2$ here.

Perhaps more importantly, remark that this bound is written in terms of  {\em unstable quantities}, namely the number of linear constraints in $Ax\leq b$ and the number of nonzero coefficients in the exact conic representation of $z^\star\in \mathcal{G}_r^{**}$. In the sections that follow, we will seek to further tighten this bound by both simplifying the coupling constraints to reduce $m$, and reducing the number of nonzero coefficients in the conic representation~\eqref{eq:sf-conic} using approximate versions of Carath\'eodory's theorem. 

\subsection{Approximate Carath\'eodory}
We can write a more stable version of the result of \citep{Aubi76} using approximate representations of the optimal solution in the Minkowski sum of epigraphs. Since the proof of Theorem~\ref{th:sf} hinges on Carath\'eodory's theorem, this will mean deriving sparser conic decompositions. The following generic result shows the impact of sparse approximate Carath\'eodory decompositions on the gap bound in~\eqref{eq:sf-bnds}, and we will seek to bound the size of these approximate representations in the sections that follow.

\begin{theorem}\label{th:gap-approx-bnd}
Suppose the functions $f_i$ in~\eqref{eq:ncvx-pb} satisfy Assumption~\ref{as:fi}. There is a point $x^\star\in\reals^d$ at which the primal optimal value of~\eqref{eq:cvx-pb} is attained, and as in~\eqref{eq:zstar} we let 
\[
z^\star = \sum_{i=1}^n
\begin{pmatrix}
f_i^{**}(x^\star_i)\\
A_ix^\star_i
\end{pmatrix}
+
\begin{pmatrix}
0\\
w-b
\end{pmatrix}
\]
with $w \in \reals^{m}_+$ be the corresponding minimizer in~\eqref{eq:dual-f}. Suppose that we use an approximate conic representation of $z^\star$ using only $s\in[n,n+m+1]$ coefficients, writing
\[
\lambda(s)=\argmin_{\substack{\lambda_{ij}\geq 0\\z_{ij}\in \mathcal{F}_i}}\left\{\left\| z^\star - \sum_{i=1}^n\sum_{j=1}^{m+2} \lambda_{ij} z_{ij}\right\|:~\sum_{i=1}^n \Card(\lambda_i)\leq s,~\ones^T\lambda_i=1, i=1,\ldots,n\right\}
\]
where $z_{ij}\in \mathcal{F}_i$ for $i=1,\ldots,n$, $j=1,\ldots,m+2$, and
$
u(s)=z^\star-\sum_{i=1}^n\sum_{j=1}^{m+2} \lambda_{ij}(s) z_{ij}.
$
with $u(s)=(u_1(s),u_2(s))$ for $u_1(s)\in\reals$ and $u_2(s)\in\reals^{m}$. We have the following bound on the solution of problem~\eqref{eq:p-ncvx-pb}
\BEQ\label{eq:approx-bnd}
\underbrace{\mathrm{h}_{CoP}(u_2(s))}_{\eqref{eq:p-cvx-pb}} ~\leq~ \underbrace{\mathrm{h}_{P}(u_2(s))}_{\eqref{eq:p-ncvx-pb}} ~\leq~ \underbrace{\mathrm{h}_{CoP}(0)}_{\eqref{eq:cvx-pb}} ~+~ \underbrace{|u_1(s)| +\max_{\beta_i\in[1,m+2]}~\left\{\sum_{i=1}^{n} \rho_{\beta_i}(f_i): \sum_{i=1}^n \beta_i = s \right\}}_{\mathrm{gap(s)}}.
\EEQ
Where $\mathrm{h}_{CoP}(\cdot)$ and $\mathrm{h}_{P}(\cdot)$ are defined in \S\ref{ss:pert} and $\rho_{k}(\cdot)$ is defined in~\eqref{def:rhok}. Furthermore, we can take $m$ to be the number of active inequality constraints at $x^\star$.
\end{theorem}
\begin{proof}
Let
$\bar z = \sum_{i=1}^n\sum_{j=1}^{m+2} \lambda_{ij}(s) z_{ij}$. By construction, this point satisfies
\[
\sum_{i=1}^n z^{(i)}_1 = \sum_{i=1}^{n} \bar z^{(i)}_1  + u_1(s) = \sum_{i=1}^{n} f_i^{**}(x_i) + u_1(s), 
\quad
\mbox{and}
\quad
\sum_{i=1}^n \bar z^{(i)}_{[2,m+1]} - b \leq u_2(s),
\]
where $z^{(i)}_{[2,m+1]}=A_ix_i^\star$. Since $f_i^{**}(x)=f_i(x)$ when $x \in \Ext(\mathcal{F}_i)$ because $\epi(f^{**})={\Co(\epi(f))}$ when Assumption~\ref{as:fi} holds, we have
\BEAS
\overbrace{\sum_{i=1}^{n} f_i^{**}(x^\star_i)}^{\ref{eq:cvx-pb}} &=& \sum_{i\in [1,n] \setminus \mathcal{S}} f_i^{**}(x_i) + \sum_{i\in \mathcal{S}} \sum_{j\in[1,m+2]} \lambda_{ij} f_i^{**}(x_{ij}) + u_1(s) \\
& = & \sum_{i\in [1,n] \setminus \mathcal{S}} f_i(x_i) + \sum_{i\in \mathcal{S}} \sum_{j\in[1,m+2]} \lambda_{ij} f_i^{**}(x_{ij}) + u_1(s) \\
& \geq & \sum_{i\in [1,n] \setminus \mathcal{S}} f_i(x_i) + \sum_{i\in \mathcal{S}} f_i^{**}\left(\tilde x _i \right)+ u_1(s) \\
& \geq & \sum_{i\in [1,n] \setminus \mathcal{S}} f_i(x_i) + \sum_{i\in \mathcal{S}} f_i\left(\tilde x _i \right) - \sum_{i\in \mathcal{S}} \rho(f_i)+ u_1(s) \\
& \geq & \underbrace{\sum_{i=1}^{n} f_i( x_i)}_{\ref{eq:p-ncvx-pb}}  ~-~ \sum_{i\in \mathcal{S}} \rho(f_i) + u_1(s) 
\EEAS
calling $\tilde x _i =\sum_{j\in[1,m+2]} \lambda_{ij} x_i$, where $\lambda_{ij}\geq 0$ and $\sum_j \lambda_{ij}=1$. The last inequality holds because the points $\tilde x_i$ are feasible for~\eqref{eq:p-ncvx-pb} with perturbation $u_2(s)$, i.e. 
\[
\sum_{i\in [1,n] \setminus \mathcal{S}} A_ix_i + \sum_{i\in \mathcal{S}} \sum_{j\in[1,m+2]} \lambda_{ij} A_i  x_{ij} \leq b + u_2(s),
\]
means that
\[
\sum_{i\in [1,n] \setminus \mathcal{S}} A_ix_i + \sum_{i\in \mathcal{S}} A_i \tilde x_i \leq b + u_2(s),
\]
which yields the desired result.
\end{proof}

The structure of this last bound differs from the previous ones because the perturbation $u$ is acting on the epigraph formulation of~\eqref{eq:p-ncvx-pb}, so it induces an error on both the objective value (the first coefficient $u_1(s)$ in this epigraph representation) and the constraints (the last $m$ coefficients $u_2(s)$). This means that we now bound the gap on a perturbed version \eqref{eq:p-ncvx-pb} of problem~\eqref{eq:ncvx-pb}, with constraint perturbation size controlled by~$u_2$. In Section~\ref{s:approx-sf} below, we will derive several approximate Carath\'eodory results to limit the size of conic representation to minimize the gap in~\eqref{eq:approx-bnd} by reducing $s$ while minimize the size of the error term $u$.

\subsection{Coupling constraints}
The tightness of the duality gap bound in~\eqref{eq:approx-bnd} depends on two distinct quantities. The first, namely $u$ discussed above, is a function of how much we can ``compress'' the convex approximation of $z^\star$ in~\eqref{eq:zstar}. The second, controlled by the sum of the nonconvexity measures $\rho_{\beta_i}(f_i)$ reads
\[
\left\{\sum_{i=1}^{n} \rho_{\beta_i}(f_i): \sum_{i=1}^n \beta_i = s \right\}
\]
and measures the severity of the problem's lack of convexity. As remarked by \citep{Udel16}, we can actually take $m$ to be the number of {\em active} constraints at the optimum of problem~\eqref{eq:ncvx-pb}, which can be substantially smaller than $m$ but is hard to bound a priori. 

The sparsity parameter $s$ above controls the tradeoff between these two components to minimize the bound, and is bounded by $n$ plus the number of active constraints. This means that another way to tighten the bound in~\eqref{eq:approx-bnd} is to reduce the number of constraints used to represent the feasible set. This number is usually a function of the {\em representation} of the feasible set, and can often be significantly reduced by transforming or approximating these representations. The results in Section~\ref{s:coupling} below will seek to make this tradeoff and all the quantities involved more explicit.

\section{Coupling Constraints}\label{s:coupling}
The duality gap bounds in~\eqref{eq:sf-bnds} or \eqref{eq:approx-bnd} heavily depend on the structure of the coupling constraints $Ax\leq b$ and exploiting this structure can lead to significant precision gain as detailed in what follows. As noticed by \citep{Udel16}, it suffices to consider only active constraints at the optimum when computing the duality gap bound in~\eqref{eq:sf-bnds} or \eqref{eq:approx-bnd}. This number can be significantly smaller than $m$. In particular, \citep[Th.\,2]{Cala05} or \citep[Lem.\,5.31]{Shap09} for example show $m\leq d$ using Helly's theorem. Bounds on the number of active constraints play a key role in solving chance constrained problems for example \citep{Cala05,Temp12,Zhan15}. Let us write
\[
A_Ix\leq b_I
\]
the equations corresponding to active constraints at the optimum, where $b_I\in\reals^{\tilde m}$. We will see below that we can further reduce the number of inequalities defining the feasible set by changing its representation or sampling them.

\subsection{Extended formulations}
The duality gap bounds in~\eqref{eq:sf-bnds} are written in terms of the number of linear constraints $Ax\leq b$ in problem~\eqref{eq:ncvx-pb}. These constraints form a polyhedron $\mathcal{P}$ and the gap bound heavily depends on the {\em representation} of this polytope. Producing a more compact formulation of $\mathcal{P}$, i.e. one using less linear inequalities, would then make our duality gap bounds much more precise. One way to produce such compact representations is to use extended formulations.

An {\em extended formulation} of the constraint polytope 
\[
\mathcal{P}=\{x\in\reals^d:Ax\leq b\}
\] 
writes it as the projection of another, potentially simpler, polytope with
\[
\mathcal{P}=\{x\in\reals^d: Bx+Cu \leq d,  u\in\reals^m\}
\]
where $B\in\reals^{q\times d}$, $C\in\reals^{q \times m}$ and $d\in\reals^{q}$. Overall, this means that we can replace $m$ in Proposition~\ref{prop:gap-sf} and Theorem~\ref{th:gap-approx-bnd} by the smallest representation size $q$ of the polytope formed by the active constraints, which can be substantially smaller than $m$ (but rarely tractable). Note however that this minimum $q$ is not the classical extension complexity of $\mathcal{P}$ discussed in the appendix because the need to keep all terms in the objective decoupled prevents us from solving and simplifying away equality constraints.

\subsection{Constraint sampling}
The results in \citep[Th.\,2]{Cala05} or \citep[Lem.\,5.31]{Shap09} show that there is an optimal solution to~\eqref{eq:cvx-pb} where at most $d+1$ constraints are active. Bounding the number of active constraints then yields better bounds on the duality gap using the results in \citep{Udel16}. Here, we show that if we allow a small approximation error from sampling constraints, we can push the number of active ones  below $d+1$, thus further improving duality gap bounds from the Shapley Folkman theorem.

The sampling results in \citep[Th.\,2]{Cala05} or \citep[Lem.\,5.31]{Shap09} use Helly's theorem to bound the number of actives constraints. We first recall a recent result by \citep{Adip19} on an approximate version of Helly's theorem.

\begin{theorem}[\citep{Adip19}]\label{th:app-helly}
Assume $K_1,\ldots,K_n$ are convex sets in an Euclidean space $\reals^d$ and let $0<k\leq n$. For $J\subset [1,n]$ define $K_J = \cap_{j\in J} K_j$. If the Euclidean unit ball $B(b,\alpha)$ centered at $b\in\reals^d$ with radius $\alpha \geq 0$ intersects $K_J$ for every $J\subset [1,n]$ with $|J|=k$, then there is a point $q\in\reals^d$ such that
\BEQ\label{eq:app-helly}
d(q,K_i)\leq \alpha \sqrt{\frac{m-k}{k(m-1)}},\quad \mbox{for $i=1,\ldots,n$}
\EEQ
where $m=\min\{n,d+1\}$.
\end{theorem}

Using this approximate Helly-type result, we now show the following result on constraint sampling for generic convex optimization problems.
\begin{theorem}\label{th:samp}
Consider
\BEQ\label{eq:cvx_gen-samp}
\BA{ll}\mbox{minimize} & f_0(x)\\
\mbox{subject to} & f_i(x)\leq 0,\quad \mbox{for $i=1,\ldots,n$,}
\EA\EEQ
in the variable $x\in\reals^d$, where the functions $f_i$ are Lipschitz continuous with constant $L_i$ with respect to the Euclidean norm. Let $0<k\leq n$, and for $J\subset [1,n]$ write
\BEQ\label{eq:cvx-samp}
\BA{rll} x_J \triangleq & \mbox{min.} & f_0(x)\\
& \mbox{s.t.} & f_j(x)\leq 0,\quad \mbox{for $j \in J$,}
\EA\EEQ
then 
\BEQ\label{eq:const-samp}
f_0(x_J) \geq f_0(x^\star) - \left(L_0+\sum_{i=1}^n \lambda_iL_i\right)\frac{D}{2}\sqrt{\frac{m-k}{k(m-1)}}
\EEQ
where $m=\min\{n,d+1\}$, $D=\diam\left(\{x\in\reals^d:f_0(x)\leq f_0(x^\star)\}\right)$ and $(x^\star,\lambda)$ are primal dual solutions to problem~\eqref{eq:cvx_gen-samp}.
\end{theorem}
\begin{proof}
By contradiction, let $\beta \geq 0$ and suppose
\[
f_0(x_J) \leq f_0(x^\star) - \beta, \quad \mbox{for all $J\subset [1,n]$ such that $|J|\leq k$.}
\]
Calling $K_0=\{x\in\reals^d:f_0(x)\leq f_0(x^\star) - \beta\}$ and $K_i=\{x\in\reals^d:f_i(x)\leq 0\}$ for $i=1,\ldots,n$, this means
\[
K_0 \cap_{j \in J}K_j \neq \emptyset.
\]
Because $K_0\subset\{ x\in\mathbb{R}^d : f(x_0) \leq f_0(x^*)\}$, there exists $b\in\mathbb{R}^d$ s.t. $K_0\subset B(b,D/2)$. Hence we have 
\[
B(b,D/2) \cap K_0 \cap_{j \in J}K_j \neq \emptyset, \quad \mbox{for all $J\subset [1,n]$ such that $|J|\leq k$.}
\]
Now Theorem~\ref{th:app-helly} shows that there is some $x_b\in\reals^d$ such that
\[
\dist(x_b,K_i) \leq \frac{D}{2} \sqrt{\frac{m-k}{k(m-1)}}, \quad \mbox{for $i=0,\ldots,n$.}
\]
We then get
\[
f_0(x_b) \leq f_0(x_0) + L_0 \|x_0 - x_b\|_2 \leq f_0(x^\star) - \beta +  \frac{L_0 D}{2} \sqrt{\frac{m-k}{k(m-1)}}
\]
for some $x_0\in K_0$, and 
\[
f_i(x_b) \leq f_i(x_i) + L_i \|x_i - x_b\|_2 \leq   \frac{L_i D}{2} \sqrt{\frac{m-k}{k(m-1)}}, \quad \mbox{for $i=1,\ldots,n$.}
\]
for some $x_i\in K_i$. Overall, this shows that
\BEQ\label{eq:pert-bnd}
\left\{\BA{l}
f_0(x_b) \leq f_0(x^\star)  +  \frac{L_0 D}{2} \sqrt{\frac{m-k}{k(m-1)}} - \beta\\
f_j(x_b) \leq \frac{L_i D}{2} \sqrt{\frac{m-k}{k(m-1)}}, \quad \mbox{for $i=1,\ldots,n$.}
\EA\right.\EEQ
Now, let us write a perturbed version of problem~\eqref{eq:cvx_gen-samp} as follows 
\[
\BA{rll}p(u)  \triangleq & \mbox{min.} & f_0(x)\\
& \mbox{s.t.} & f_i(x)\leq u,\quad \mbox{for $i=1,\ldots,n$,}
\EA\]
in the variable $x\in\reals^d$ and perturbation parameters $u\in\reals^n$. The inequalities on $x_b$ in~\eqref{eq:pert-bnd} imply that 
\[
p(u) \leq f_0(x^\star) + \frac{L_0 D}{2} \sqrt{\frac{m-k}{k(m-1)}} - \beta
\]
where
\[
u_i=\frac{L_i D}{2} \sqrt{\frac{m-k}{k(m-1)}}, \quad \mbox{for $i=1,\ldots,n$.}
\]
Weak duality (see e.g. \citep[\S5.6.2]{Boyd03}) also shows that
\[
p(u) \geq f_0(x^\star) -\lambda^T u
\]
and we get a contradiction unless
\[
f_0(x^\star) + \frac{L_0 D}{2} \sqrt{\frac{m-k}{k(m-1)}} - \beta \geq f_{0}(x^\star) -\lambda^T u
\]
which is the desired result.
\end{proof}

Note that when $D=\infty$ the statement in \eqref{eq:const-samp} is of course vacuous. In particular additional assumptions on $f$, such as strong convexity, would ensure that $D$ is finite and improve the factor $D$ in the right hand side of \eqref{eq:const-samp}. In practice, to minimize the number of active constraints, we typically look for sparse solutions to the dual. Theorem~\ref{th:samp} quantifies the tradeoff between number of constraints and duality gap for approximately sparse solutions. This tradeoff will be especially favorable if the quantity 
\[
\left(L_0+\sum_{i=1}^n \lambda_iL_i\right).
\]
which can be seen as the weighted $\ell_1$ norm of the dual solution is small. This quantity thus explicitly quantifies the impact of constraint sampling for approximately sparse solutions.


\section{Approximate Carath\'eodory \& Shapley-Folkman Theorems}\label{s:approx-sf}
We will now derive a version of the Shapley-Folkman result in Theorem~\ref{th:sf} which only approximates~$x$ but where $\mathcal{S}$ is typically smaller.

\subsection{Approximate Carath\'eodory Theorems}
Recent activity around Carath\'eodory's theorem \citep{Dona97,Vers12,Dai14} has focused on producing tight approximate versions of this result, where one aims at finding a convex combination using fewer elements, which is still a ``good'' approximation of the original element of the convex hull. The following theorem for instance, produces an upper bound on the number of elements needed to achieve a given level of precision, using a randomization argument.
\begin{theorem}[Approximate Carath\'eodory]\label{th:approx-cara}
  Let $V \subset \mathbb{R}^d$, $x \in \Co(V)$ and $\varepsilon > 0$. We assume that $V$ is bounded and we write $D_p$ the quantity $D_p \triangleq \sup_{v \in V} \| v \|_p$, for $p \geq 2$. Then, there exists some $\hat{x}\in \Co(V)$ and $m \leq 8pD_p^2/\varepsilon^2$ such that
  \[
\textstyle    \|x - \hat{x} \| = \left\|x - \sum_{i = 1}^m \lambda_i v_i \right\|_p \leq \varepsilon,
  \]
  for some $v_i \in V$, $\lambda_i > 0$ and $\ones^{\top} \lambda = 1$.
\end{theorem}
This result is a direct consequence of Maurey's lemma \citep{Pisi81} and is based on a probabilistic approach which samples vectors $v_i$ with replacement and uses concentration inequalities to control approximation error, but can also be seen as a direct application of Frank-Wolfe type algorithms to 
\[
\textstyle \underset{v\in\Co(V)}{\mbox{minimize}} ~~ \left\|x - v \right\|~,
\]
where the algorithm is stopped when the iterate has enough extreme points in its representation.

In the results that follow however, we will have $N=n+m+1$, and we will seek approximations using $s$ terms with $s\in[n,n+m+1]$ with $n$ typically much bigger than~$m$. Sampling with replacement does not provide precise enough bounds in this setting and we will use results from \citep{Serf74} on sample sums {\em without replacement} to produce a more precise version of the approximate Carath\'eodory theorem that handles the case where a high fraction of the coefficients is sampled.

\begin{theorem}[High-sampling ratio in $\ell_{\infty}$]\label{th:approx-cara-high}
Let $x=\sum_{j=1}^N \lambda_j V_j {\in\mathbb{R}^d}$ for $V \in \mathbb{R}^{d\times N}$ and some $\lambda \in \reals^N$ such that $\ones^T\lambda=1, \lambda \geq 0$. Let $\varepsilon > 0$ and write $R=\max\{ R_v,R_\lambda\}$ where $R_v=\max_{i} \|\lambda_i V_i\|_\infty$ and  $R_\lambda=\max_{i} |\lambda_i| $. Then, there exists some $\hat{x}=\sum_{j\in \mathcal{J}} \mu_j V_j$ with $\mu \in \reals^m$ and $\mu \geq 0$, where $\mathcal{J} \subset [1,N]$ has size
\[
|\mathcal{J}| = 1 + N\,\frac{{2\log(4d)(\sqrt{N}\,R/\varepsilon)^2}}{1+2 {\log(4d)(\sqrt{N}\,R/\varepsilon)^2}}
\] 
and is such that $\|x - \hat{x} \|_\infty \leq \varepsilon$ and $|\sum_{j \in \mathcal{J}} \mu_j -1|\leq \varepsilon$.
\end{theorem}
\begin{proof}
Let $\epsilon>0$ and
\[
S^{(i)}_m =\sum_{j\in\mathcal{J}} \lambda_j V_j^{(i)}
\]
where $\mathcal{J}$ is a random subset of $[1,N]$ of size $m$, then \citep[Cor\,1.1]{Serf74} shows
\[
\Prob\left(\left|\frac{N}{m} S^{(i)}_m - x^{(i)} \right|\geq \varepsilon \right) \leq 2 \exp\left(\frac{-\alpha_m \varepsilon^2}{2N(1-\alpha_m){R^2}}\right)
\]
where $\alpha_m=(m-1)/N$ is the sampling ratio. Let $\beta\in]0,1[$, a union bound then means that setting 
\BEQ\label{eq:choice_ratio}
\frac{\alpha_m}{1-\alpha_m} \geq \frac{2 \log(2 d/(1 - \beta))(\sqrt{N}\,R)^2}{\varepsilon^2}
\EEQ
ensures $\|x - \hat{x} \|_\infty \leq \varepsilon$ with probability at least $\beta$. With $\alpha_m$ as in \eqref{eq:choice_ratio}, a similar reasoning applied to $\mu = \frac{N}{m}\lambda$, with $\mathcal{S}$ a random subset of $[1,N]$ of size $m$, ensures that $|\sum_{j \in \mathcal{S}} \mu_j -1|\leq \varepsilon$ with probability at least $1-(1-\beta)/d$. Hence any choice of $\beta> 1/(d+1)$ ensures that there exists a subset $\mathcal{J}$ of size $m$ with
\[
\alpha_m \geq  \frac{{2\log(d/(1-\beta))(\sqrt{N}\,R/\varepsilon)^2}}{1+{2\log(d/(1-\beta))(\sqrt{N}\,R/\varepsilon)^2}}
\]
that yields the desired result (with $\beta=1/2$ for instance).
\end{proof}

The result above uses Hoeffding-Serfling bounds on real-valued random variables to provide error bounds in $\ell_\infty$ norm. 
Since the vectors we consider here have a block structure coming the epigraphs $\mathcal{F}_i$, we consider generic Banach spaces to properly fit the norm to this structure by extending this last result to arbitrary norms in $(2,D)$-smooth Banach spaces using a recent result on Hoeffding-Serfling bounds in Banach spaces by \citep{Schn16}.

\begin{theorem}[Approximate Carath\'eodory with High Sampling Ratio in Banach spaces]\label{th:approx-cara-high-banach}
Let $x=\sum_{j=1}^N \lambda_j V_j\in \mathbb{R}^d$ for $V \in \mathbb{R}^{d\times N}$ and some $\lambda \in \reals^N$ such that $\ones^T\lambda=1, \lambda \geq 0$. Let $\varepsilon > 0$ and write $R=\max\{D R_v,R_\lambda\}$ where $R_v=\max_{i} \|\lambda_i V_i\|$ and  $R_\lambda=\max_{i} |\lambda_i|$, for some norm $\|\cdot\|$ such that $(\mathbb{R}^{d},\|\cdot\|)$ is $(2,D)$-smooth (see Definition \ref{def:2D_smooth}). Then, there exists some $\hat{x}=\sum_{j\in \mathcal{J}} \mu_j V_j$ with $\mu \in \reals^m$ and $\mu \geq 0$, where $\mathcal{J} \subset [1,N]$ has size
\BEQ\label{eq:size_J}
|\mathcal{J}| =1+  N\,  \frac{{c(\sqrt{N}\,R/\varepsilon)^2}}{1+{c(\sqrt{N}\,R/\varepsilon)^2}}
\EEQ
for some absolute constant $c>0$, and is such that $\|x - \hat{x} \| \leq \varepsilon$ and $|\sum_{j \in \mathcal{J}} \mu_j -1|\leq \varepsilon$.
\end{theorem}
\begin{proof}
We use \citep[Th.\,1]{Schn16} instead of \citep[Cor\,1.1]{Serf74} in the proof of Theorem~\ref{th:approx-cara-high}. This means imposing
\[
\alpha_m \geq  \frac{{c(\sqrt{N}\,R/\varepsilon)^2}}{1+{c(\sqrt{N}\,R/\varepsilon)^2}}
\]
Finally, $R=\max\{D R_v,R_\lambda\}\geq R_\lambda$ ensures that the Hoeffding like bound in \citep{Serf74} also holds, with $|\sum_{j \in \mathcal{S}} \mu_j -1|\leq \varepsilon$, and yields the desired result.
\end{proof}

For the sake of clarity, the result of the Theorem \ref{th:approx-cara-high-banach} is spelled out as deterministic. In fact however, it states that there is a probability (related to a particular value of $c$) that a uniformly chosen subset $\mathcal{J}\subset[1,N]$ of size \eqref{eq:size_J} satisfies the Approximate Carath\'eodory conditions. We use this probabilistic version in the proof of Theorem \ref{th:approx-sf-th}.

For real valued random variables, recent results by \citep{Bard15} provide Bernstein-Serfling type inequalities where the radius $R$ above can be replaced by a standard deviation. This leads to an extension of Theorem \ref{th:approx-cara-high}. We also show a Bennett-Serfling inequality in Banach spaces in \S\ref{ssec:Bennet_Sterfling} which allows us to control the sampling ratio using a variance term. This means that the sampling ratio in Theorem~\ref{th:approx-cara-high} above can be replaced by 
\[
\alpha_m \geq \frac{c \big[ 2(D\sigma_m)^2 +\epsilon R_v/(3N)\big]N}{\epsilon^2 + c \big[2(D\sigma_m)^2\big]N},
\]
where
\[
\sigma_m \triangleq \frac{1}{\sqrt{\sum_{k=1}^{m}{\frac{1}{(N-k)^2}}}}\Big\vert\Big\vert  \Big( \sum_{k=1}^{m}{\frac{1}{(N-k)^2} \mathbb{E}_{k-1}||V_k-\mathbb{E}_{k-1}(V_k)||^2\Big)^{1/2}} \Big\vert\Big\vert_{\infty}~,
\]
plays the role of the standard deviation when sampling without replacement. We call $\sigma_m^2$ a variance because it is the essential supremum of a convex combination of the terms $\tilde{\sigma}_k^2 = \mathbb{E}_{k-1}||V_k-\mathbb{E}_{k-1}(V_k)||^2$ (see \eqref{eq:def_filtration} for definition of $\mathbb{E}_{k-1}$). For $k=1$, $\sigma_1^2$ is exactly the variance of $V$, while when $k=N-1$, $\tilde{\sigma}_k$ is not much different from the diameter of the set $V$.

\subsection{Approximate Shapley-Folkman Theorems}
We now prove an approximate version of the Shapley-Folkman theorem, plugging approximate Carath\'eodory results inside the proof of Theorem~\ref{th:sf}.

\begin{theorem}[Approximate Shapley-Folkman]\label{th:approx-sf-th}
Let $\varepsilon,\beta,\gamma >0$ and $V_i \in\reals^d$, $i=1,\ldots,n$ be a family of subsets of $\reals^d$. Suppose
\[
x = \sum_{i=1}^n\sum_{j=1}^{d+1} \lambda_{ij} v_{ij} ~\in ~\sum_{i=1}^n\Co\left(V_i\right)
\]
where $\lambda_{ij} \geq 0$ and $\sum_j \lambda_{ij}=1$. We write $R=\max\{\beta D R_v,  \gamma R_\lambda\}$ where $R_v= \max_{\{ij:\lambda_{ij}\neq 1\}} \|\lambda_{ij} v_{ij}\|$ and  $R_\lambda= \max_{\{ij:\lambda_{ij}\neq 1\}} |\lambda_{ij}|$, for some norm $\|\cdot\|$ such that $(\mathbb{R}^{d},\|\cdot\|)$ is $(2,D)$-smooth. Then there exists a point $\hat x\in\reals^d$, coefficients $\mu_{i} \geq 0$ and index sets $\mathcal{S},\mathcal{T}\subset[1,n]$ with $\mathcal{S}\cap\mathcal{T} =\emptyset$ such that $q\triangleq |\mathcal{S}|+|\mathcal{T}| \leq d$, and
\[
\textstyle \hat x \in~ \sum_{[1,n]\setminus (\mathcal{S}\cup\mathcal{T})} V_i ~ + ~  \sum_{i \in \mathcal{T}} \mu_i V_i ~ + ~  \sum_{i \in \mathcal{S}} \mu_i \Co(V_i)
\]
with
\[
\|x-\hat x\| \leq \frac{q}{\beta} \varepsilon,
\qquad
\left|\sum_{i\in\mathcal{S}\cup\mathcal{T}} \mu_{i} - q\right|\leq q\varepsilon
\qquad
\mbox{and}
\qquad
\left(\sum_{i\in\mathcal{S}\cup\mathcal{T}}\left(\mu_{i} - 1\right)^2\right)^{1/2} \leq \frac{q}{\gamma} \varepsilon.
\]
where $|\mathcal{S}|\leq (m - |\mathcal{T}|)/2$ with
\BEQ\label{eq:def_m}
m {\triangleq} 1+  (d+q)\,\frac{c(\sqrt{d+q}\,R/q\varepsilon)^2}{1+c(\sqrt{d+q}\,R/q\varepsilon)^2}.
\EEQ
hence, in particular, $|\mathcal{S}|\leq m - q$.
\end{theorem}
\begin{proof}
If $x\in \sum_{i=1}^n\Co\left(V_i\right)$, as in the proof of Theorem~\ref{th:sf} above, we can write 
\[
z=\sum_{i=1}^n\sum_{j=1}^{d+1} \lambda_{ij} z_{ij}
\]
where $z\in\reals^{d + n}$ and
\[
z=
\begin{pmatrix}
\beta x \\ \gamma\ones_{n}
\end{pmatrix}
,\quad
z_{ij}=
\begin{pmatrix}
\beta v_{ij} \\ \gamma e_i
\end{pmatrix}
,\quad\mbox{for $i=1,\ldots,n$ and $j=1,\ldots,d+1$,}
\]
with $e_i \in \reals^n$ is the Euclidean basis, $\gamma,\beta>0$ and by the classical Carath\'eodory bound, at most $d + n$ coefficients $\lambda_{ij}$ are nonzero (note the extra scaling factors $\gamma,\beta>0$ here compared to Theorem~\ref{th:sf}). Let us call $\mathcal{I}\subset [1,n]$ the set of indices such that $i\in\mathcal{I}$ iff at least two coefficients in $\{\lambda_{ij}: j\in[1,d+1]\}$ are nonzero. As in Theorem~\ref{th:sf}, we must have $|\mathcal{I}|\leq d$. We write
\[
\frac{y}{|\mathcal{I}|}=\sum_{i\in\mathcal{I}}\sum_{j=1}^{d+1} \frac{\lambda_{ij}}{|\mathcal{I}|} z_{ij}
\]
where $\sum_{i\in\mathcal{I}}\sum_{j=1}^{d+1} \lambda_{ij}/|\mathcal{I}|=1$ and at most $d+|\mathcal{I}|$ coefficients $\lambda_{ij}$ are nonzero. We will apply the probabilistic version of Theorem~\ref{th:approx-cara-high-banach} twice here with radius $R/q$ where $q=|\mathcal{I}|$. Once on the upper block of the vectors $z_{ij}$ using the norm $\|\cdot\|$ and then on the lower blocks of these vectors (corresponding to the constraints on $\lambda_{ij}$), using the $\ell_2$ norm to exploit the fact that these lower blocks have comparatively low $\ell_2$ radius. 

Theorem~\ref{th:approx-cara-high-banach} applied to the upper block of $y/q$ and of the vectors $z_{ij}$ shows that with probability higher than $1/2$ there exists some ${\hat{x}}/{|\mathcal{I}|}=\sum_{i\in\mathcal{I}}\sum_{j=1}^{d+1} \mu_{ij} v_{ij}$ with $|\sum_{i\in\mathcal{I}}\sum_{j=1}^{d+1} \mu_{ij} -1|\leq\varepsilon$,  $\mu \geq 0$, where at most $m$ (defined in \eqref{eq:def_m}) coefficients $\mu_{ij}$ are nonzero and 
\[
\left\| x - \sum_{i\in[1,n]\setminus \mathcal{I}} v_i - \hat{x} \right\| \leq |\mathcal{I}|\varepsilon/\beta.
\] 
for some $v_i\in V_i$. Then, Theorem~\ref{th:approx-cara-high-banach} applied to the lower block of the vectors $z_{ij}$ shows that with probability higher than $1/2$ the weights $\mu_{ij}$ sampled above satisfy 
\[
\textstyle \left(\sum_{i\in\mathcal{I}}\left(\sum_{j=1}^{d+1} |\mathcal{I}| \mu_{ij} - 1\right)^2\right)^{1/2} \leq |\mathcal{I}|\varepsilon / \gamma.
\]
with the $\ell_2$ norm being $D=1$ smooth. Let $\mathcal{T}$ be the equivalent of $\mathcal{I}$ with respect to the sequence $(\mu_{i,j})$. Setting $\mathcal{I}=\mathcal{S}\cup\mathcal{T}$, and since $m$ nonzero coefficients are spread among $q$ sets, we have $|\mathcal{S}| \leq m -q$. Finally writing $\mu_i= \sum_j |\mathcal{I}|\mu_{ij}$ then yields the desired result.
\end{proof}

We then have the following corollary, producing a simpler instance of the previous theorem.

\begin{corollary}\label{th:approx-sf}
Let $\varepsilon >0$ and $V_i \in\reals^d$, $i=1,\ldots,n$ be a family of subsets of $\reals^d$. Suppose
\[
x = \sum_{i=1}^n\sum_{j=1}^{d+1} \lambda_{ij} v_{ij} ~\in ~\sum_{i=1}^n\Co\left(V_i\right)
\]
where $\lambda_{ij} \geq 0$ and $\sum_j \lambda_{ij}=1$. We write $R_v= \max_{\{ij:\lambda_{ij}\neq 1\}} \|\lambda_{ij} v_{ij}\|$ and  $R_\lambda= \max_{\{ij:\lambda_{ij}\neq 1\}} |\lambda_{ij}|$, for some norm $\|\cdot\|$ such that $(\mathbb{R}^{d},\|\cdot\|)$ is $(2,D)$-smooth. There exists a point $\bar x$ and an index set $\mathcal{S}\subset[1,n]$ such that 
\[
\bar x \in~ \sum_{[1,n]\setminus \mathcal{S}} V_i  ~ + ~  \sum_{i \in \mathcal{S}} \Co(V_i)
\quad
\mbox{with}
\quad
\|x-\bar x\| \leq  \sqrt{2d}\, \left(\frac{R_v}{R_\lambda} + M_V\right)  \varepsilon
\]
where $|\mathcal{S}|\leq  m - d$ with
\BEQ\label{eq:samp-ratio}
m=1+  2d\,\frac{c\,(D R_\lambda/\varepsilon)^2}{1+c\,(D R_\lambda/\varepsilon)^2}
\qquad
\mbox{and}
\qquad
M_V=\sup_{\substack{\|u\|_2\leq1\\v_i \in V_i}} \left\|\sum_i u_i v_i \right\|.
\EEQ
where $c>0$ is an absolute constant.
\end{corollary}
\begin{proof} Theorem~\ref{th:approx-sf-th} means there exists $\hat x\in\reals^d$, coefficients $\mu_i\geq 0$ and index sets $\mathcal{S},\mathcal{T}\subset[1,n]$ such that
\BEAS
\hat x &\in&  \sum_{[1,n]\setminus (\mathcal{S}\cup\mathcal{T})} V_i ~ + ~  \sum_{i \in \mathcal{T}} \mu_{i} V_i ~ + ~  \sum_{i \in \mathcal{S}} \mu_{i} \Co(V_i)\\
{ }\\
&\subset& \sum_{[1,n]\setminus \mathcal{S}} V_i ~ + ~  \sum_{i \in \mathcal{S}} \Co(V_i) ~ + ~ \sum_{i \in \mathcal{T}} \left(\mu_{i}-1\right) V_i ~ + ~  \sum_{i \in \mathcal{S}} \left(\mu_{i}-1\right) \Co(V_i)\\
\EEAS
with 
\[
\textstyle \left(\sum_{i\in\mathcal{I}}\left(\mu_{i} - 1\right)^2\right)^{1/2} \leq q\; \varepsilon / \gamma.
\qquad
\mbox{and}
\qquad
\|x-\hat x\| \leq q\; \varepsilon/ \beta
\]
where $q\triangleq |\mathcal{S}|+|\mathcal{T}| \leq d$. Saturating the max term in $R$ in Theorem~\ref{th:approx-cara-high-banach} means setting $\beta R_v =\gamma R_\lambda$. Setting $\gamma=q/\sqrt{d+q}$ then yields
$\|x-\hat x\| \leq \sqrt{d+q} \frac{R_v}{R_\lambda} \varepsilon$ and 
\[
\textstyle \left(\sum_{i\in\mathcal{I}}\left(\mu_{i} - 1\right)^2\right)^{1/2} \leq \sqrt{d+q} \,\varepsilon.
\]
and the fact that 
\[
v \in \sum_{i \in \mathcal{T}} \left(\mu_{i}-1\right) V_i ~ + ~  \sum_{i \in \mathcal{S}} \left(\mu_{i}-1\right) \Co(V_i)
\]
means 
\[
\textstyle \|v\|\leq M_V  \left(\sum_{i\in\mathcal{I}}\left(\mu_{i} - 1\right)^2\right)^{1/2}
\]
and yields the desired result.
\end{proof}

The result of \citet{Aubi76} recalled in Proposition~\ref{prop:gap-sf} shows that the Shapley-Folkman theorem can be used in the bounds of Proposition~\ref{prop:gap} to ensure the set $\mathcal{S}$ is of size at most $m+1$, therefore providing an upper bound on the duality gap caused by the lack of convexity (see also \citep{Ekel99,Bert14}). We now study what happens to these bounds when use use the approximate Shapley-Folkman result in Corollary~\ref{th:approx-sf} instead of Theorem~\ref{th:sf}. Plugging these last results inside the duality gap bound in Theorem~\ref{th:gap-approx-bnd} yields the following result.

\begin{corollary}\label{th:gap-approx-sf}
Suppose the functions $f_i$ in~\eqref{eq:ncvx-pb} satisfy Assumption~\ref{as:fi}. There is a point $x^\star\in\reals^d$ at which the primal optimal value of~\eqref{eq:cvx-pb} is attained, and as in~\eqref{eq:zstar} we let 
\[
z^\star = \sum_{i=1}^n
\begin{pmatrix}
f_i^{**}(x^\star_i)\\
A_ix^\star_i
\end{pmatrix}
+
\begin{pmatrix}
0\\
w-b
\end{pmatrix}
=
\sum_{i=1}^n \sum_{j=1}^{m+2} \lambda_{ij}z_{ij} 
+
\begin{pmatrix}
0\\
w-b
\end{pmatrix}
\]
with $w \in \reals^{m}_+$ and $z_{ij}\in \mathcal{F}_i$, where $\lambda_{ij}\geq0$, $\sum_j \lambda_{ij}=1$. Call $R_v= \max_{\{ij:\lambda_{ij}\neq 1\}} \|\lambda_{ij} z_{ij}\|_2$ and $R_\lambda= \max_{\{ij:\lambda_{ij}\neq 1\}} |\lambda_{ij}|$. Let $\gamma >0$, we have the following bound on the solution of problem~\eqref{eq:p-ncvx-pb}
\[
\underbrace{\mathrm{h}_{CoP}(u_2(s))}_{\eqref{eq:p-cvx-pb}} ~\leq~ \underbrace{\mathrm{h}_{P}(u_2(s))}_{\eqref{eq:p-ncvx-pb}} ~\leq~ \underbrace{\mathrm{h}_{CoP}(0)}_{\eqref{eq:cvx-pb}} ~+~ \underbrace{|u_1(s)| +\max_{\beta_i\in[1,m+2]}~\left\{\sum_{i=1}^{n} \rho_{\beta_i}(f_i): \sum_{i=1}^n \beta_i = s \right\}}_{\mathrm{gap(s)}}.
\]
where
\BEQ\label{eq:approx-u}
\max\{|u_1(s)|,\|u_2(s)\|_2\} \leq\sqrt{2m}\, \left(R_v + R_\lambda M_V\right)  \gamma
\EEQ
with
\[
s=n+1+2m\,\frac{c}{\gamma^2+c}
\quad\mbox{and}
\quad
M_V=\sup_{\substack{\|u\|_2\leq1\\v_i \in \mathcal{F}_i}} \left\|\sum_i u_i v_i \right\|_2,
\]
for some absolute constant $c>0$.
\end{corollary}
\begin{proof}
This is a direct consequence of Corollary~\ref{th:approx-sf}.
\end{proof}

Once again, we can take $m$ to be the number of active inequality constraints at $x^\star$. Note that in practice, not all solutions $z^*$ are good starting points for the approximation result described above. Obtaining a good solution typically involves a ``purification step'' along the lines of \citep{Udel16} for example.

\section{Separable Constrained Problems}\label{s:bnds-const}
Here, we briefly show how to extend our previous to problems with separable {\em nonlinear} constraints. We now focus on a more general formulation of optimization problem~\eqref{eq:ncvx-pb}, written
\BEQ\label{eq:ncvx-pb-const}\tag{cP}
\BA{ll}
\mbox{minimize} & \sum_{i=1}^{n} f_i(x_i) \\
\mbox{subject to} & \sum_{i = 1}^n g_i(x_i) \leq b,\\
& x_i \in Y_i, \quad i=1,\ldots,n,
\EA\EEQ
where the $g_i$'s take values in $\reals^m$. We assume that the functions $g_i$ are lower semicontinuous. Since the constraints are not necessarily affine anymore, we cannot use the convex envelope to derive the dual problem. The dual now takes the generic form
\BEQ\label{eq:dual-pb-const}\tag{cD}
\sup_{\lambda \geq 0} \Psi(\lambda),
\EEQ
where $\Psi$ is the dual function associated to problem~\eqref{eq:ncvx-pb-const}. Note that deriving this dual explicitly may be hard. As for problem~\eqref{eq:ncvx-pb}, we will also use the perturbed version of problem~\eqref{eq:ncvx-pb-const}, defined as
\BEQ\label{eq:p-ncvx-pb-const}\tag{p-cP}
\BA{rll}
\mathrm{h}_{cP}(u) \triangleq &\mbox{min.} & \sum_{i=1}^{n} f_i(x_i) \\
&\mbox{s.t.} & \sum_{i = 1}^n g_i(x_i) - b \leq u\\
&& x_i \in Y_i, \quad i=1,\ldots,n,
\EA\EEQ
in the variables $x_i\in\reals^{d_i}$, with perturbation parameter $u\in\reals^m$. We let $\mathrm{h}_{cD} \triangleq \mathrm{h}_{cP}^{**}$ and in particular, solving for $\mathrm{h}_{cD}(0)$ is equivalent to solving problem~\eqref{eq:dual-pb-const}. Using these new definitions, we can formulate a more general bound for the duality gap (see \cite[Appendix I, Thm. 3]{Ekel99} for more details).
\begin{proposition}
  Suppose the functions $f_i$ and $g_i$ in~\eqref{eq:ncvx-pb-const} are such that all $(f_i + \ones^{\top} g_i)$ satisfy Assumption~\ref{as:fi}. Then, one has
\[
  \mathrm{h}_{cD}((m + 1)\bar{\rho}_g) \leq \mathrm{h}_{cP}((m + 1) \bar{\rho}_g) \leq \mathrm{h}_{cD}(0) + (m + 1) \bar{\rho}_f,
\]
where $\bar{\rho}_f = \sup_{i \in [1,n]} \rho(f_i)$ and $\bar{\rho}_g = \sup_{i \in [1, n]} \rho(g_i)$.
\end{proposition}
\begin{proof}
The global reasoning is similar to Proposition~\ref{prop:gap-sf}, using the graph of $\mathrm{h}_{cP}$ instead of the $\mathcal{F}_i$'s. 
\end{proof}

We then get a direct extension of Corollary~\ref{th:gap-approx-sf}, as follows.

\begin{corollary}\label{cor:gap-approx-sf}
Suppose the functions $f_i$ and $g_i$ in~\eqref{eq:ncvx-pb-const} are such that all $(f_i + \ones^{\top} g_i)$ satisfy Assumption~\ref{as:fi}. There exist points $x^\star_{ij}\in\reals^{d_i}$ and $w \in \reals^m$ such that
  \[
    z^{\star} = \sum_{i = 1}^n \sum_{j = 1}^{m + 2} \lambda_{ij} (f_i(x^{\star}_{ij}), g_i(x^{\star}_{ij})) + (0, -b + w),
  \]
attains the minimum in~\eqref{eq:dual-pb-const}, where $\lambda_{ij}\geq0$ and $\sum_j \lambda_{ij}=1$. Call $R_v= \max_{\{ij:\lambda_{ij}\neq 1\}} \|\lambda_{ij} z_{ij}\|_2$ and $R_\lambda= \max_{\{ij:\lambda_{ij}\neq 1\}} |\lambda_{ij}|$. Let $\gamma >0$, we have the following bound on the solution of problem~\eqref{eq:ncvx-pb-const}
\BEAS
\underbrace{\mathrm{h}_{cD}(u_2(s)+(m+1)\bar{\rho}_g\ones)}_{\eqref{eq:dual-pb-const}} &\leq& \underbrace{\mathrm{h}_{P}(u_2(s)+(m+1)\bar{\rho}_g\ones)}_{\eqref{eq:p-ncvx-pb-const}} \\
&\leq& \underbrace{\mathrm{h}_{cD}(0)}_{\eqref{eq:dual-pb-const}} ~+~ \underbrace{|u_1(s)| +\max_{\beta_i\in[1,m+2]}~\left\{\sum_{i=1}^{n} \rho_{\beta_i}(f_i): \sum_{i=1}^n \beta_i = s \right\}}_{\mathrm{gap(s)}}.
\EEAS
where $\bar{\rho}_g = \sup_{i \in [1, n]} \rho(g_i)$ and
\[
\max\{|u_1(s)|,\|u_2(s)\|_2\} \leq\sqrt{2m}\, \left(R_v + R_\lambda M_V\right)  \gamma
\]
with
\[
s=n+1+2m\,\frac{c}{\gamma^2+c}
\quad\mbox{and}
\quad
M_V=\sup_{\substack{\|u\|_2\leq1\\v_i \in \mathcal{F}_i}} \left\|\sum_i u_i v_i \right\|_2,
\]
for some absolute constant $c>0$.
\end{corollary}

For simplicity, we have used coarse bounds on $\rho(g_i)$ but these can be relaxed to stable quantities using techniques matching those used on the objective in the previous sections.

\section*{Acknowledgements}
AA is at CNRS \& d\'epartement d'informatique, \'Ecole normale sup\'erieure, UMR CNRS 8548, 45 rue d'Ulm 75005 Paris, France,  INRIA and PSL Research University. The authors would like to acknowledge support from the {\em Optimization \& Machine Learning} joint research initiative with the {\em fonds AXA pour la recherche} and Kamet Ventures as well as a Google focused award. The authors would also like to thank Alessandro Rudi and Carlo Ciliberto for very helpful discussions on multitask problems.

{\small \bibliographystyle{plainnat}
\bibsep 1ex
\bibliography{MainPerso.bib,refs}}

\section{Appendix}
We now recall some background results on extended formulations and Bennett-Sterfling inequalities smooth Banach spaces.

\subsection{Extended formulations}
An {\em extended formulation} of the constraint polytope 
\[
\mathcal{P}=\{x\in\reals^d:Ax\leq b\}
\] 
writes it as the projection of another, potentially simpler, polytope with
\[
\mathcal{P}=\{x\in\reals^d: Bx+Cu \leq d,  u\in\reals^m\}
\]
where $B\in\reals^{q\times d}$, $C\in\reals^{q \times m}$ and $d\in\reals^{q}$. The {\em extension complexity} $xc(\mathcal{P})$ is the minimum number of inequalities of an extended formulation of the polytope $\mathcal{P}$. A fundamental result by \citep[Th.\,3]{Yann91} connects extended formulations and nonnegative matrix factorization. Suppose the vertices of a polytope $\mathcal{P}=\{x\in\reals^d:Ax\leq b\}$ are given by $\{v_1,\ldots,v_p\}$, we write $S$ the {\em slack matrix} of $\mathcal{P}$, with 
\[
S_{ij}=b_i-(Av_j)_i,\quad\mbox{for $i=1,\ldots,m$, $j=1,\ldots,p$.}
\]
By construction, $S$ is a nonnegative matrix. \citep[Th.\,3]{Yann91} shows that 
\[
\{x\in\reals^d:Ax+Fy=b,y\geq 0\}
\] 
is an extended formulation of $\mathcal{P}$ if and only if $S$ can be factored as $S=FV$ where $F\in\reals_+^{m\times q}$ and $V\in\reals_+^{q\times p}$ are both nonnegative. In particular, the smallest extended formulation of $\mathcal{P}$ corresponds to the lowest rank NMF of $S$, which means $xc(\mathcal{P})=\Rank_+(S)$, the nonnegative rank of $S$. 

Usually, the smallest representation is then found by simplifying away the equality constraints in $Ax+Fy=b$ to keep only linear inequalities. This last operation cannot work in our context here since it might introduce coupled variables in the sum of terms in the objective. So the best representation in the context of this pape does not correspond to the classical one with smallest extension complexity.

While the nonnegative rank is again an unstable quantity, stable (approximate) versions of this result can be defined using {\em nested polytopes} \citep{Pash12,Brau12,Gill12}. Given polytopes $\mathcal{P},\mathcal{Q}\in\reals^d$, an extended formulation of the pair $(\mathcal{P},\mathcal{Q})$ is a polytope 
\[
\mathcal{K}=\max\{x \in\reals^d : Ax+Fy=b,y\geq 0\}
\]
such that $\mathcal{P}\subset \mathcal{K} \subset \mathcal{Q}$. Furthermore, suppose $\mathcal{P}=\Co(\{v_1,\ldots,v_p\})$ and $\mathcal{Q}=\{x\in \reals^d : Ax\leq b \}$, defining the slack matrix of the pair $(\mathcal{P},\mathcal{Q})$ as $S_{ij}=b_i-(Av_j)_i$, for $i=1,\ldots,m$, $j=1,\ldots,p$, the result in \cite[Th.\,1]{Brau12} shows that the extension complexity of the pair satisfies $xc(\mathcal{P},\mathcal{Q})\leq\Rank_+(S)+1$.

\subsection{Bennett-Sterfling Inequalities in (2,D) Smooth Banach Spaces}\label{ssec:Bennet_Sterfling}
We prove a Bennett-Sterfling inequality in Theorem \ref{th:Bennet_Sterfling} below. This concentration inequality allows to rewrite the bound involving the quantity $R$ in Theorem \ref{th:approx-cara-high} with a term taking into account the variance of $V$, hence leading to an approximate Carath\'eodory version for high sampling ratio and low variance.

Consider $V=\{\vv_1;\ldots;\vv_N\}$, a set of $N$ vectors in a $(2,D)-$Banach space with norm $||\cdot||$ and $V_1,\ldots, V_n,$ the random variables resulting from 	a sampling without replacement. $ R_v \triangleq \sup_i ||\vv_i||$ is the \textit{range} of $V$. We introduce a specific notion of variance related to that sampling scheme as follows
\BEQ\label{def:variance_without_replacement}
\sigma_m \triangleq \frac{1}{\sqrt{\sum_{k=1}^{m}{\frac{1}{(N-k)^2}}}}\Big\vert\Big\vert  \Big( \sum_{k=1}^{m}{\frac{1}{(N-k)^2} \mathbb{E}_{k-1}||V_k-\mathbb{E}_{k-1}(V_k)||^2\Big)^{1/2}} \Big\vert\Big\vert_{\infty}~,
\EEQ
where we write $||\cdot||_{\infty}$ for essential supremum. We identify it as a variance because it is a convex combination of the terms $\mathbb{E}_{k-1}||V_k-\mathbb{E}_{k-1}(V_k)||^2$. For $k=1$, it is exactly the variance of $V$, while when $k=N-1$ it is not much different from the diameter of the set $V$. This is the natural notion algebraically arising from the sampling without replacement. Nevertheless, one can notice that when the index $k$ increases the weights also do, thus putting more emphasis on diameter-like measures rather than on variance-like measures.

Our goal is to bound, the following probability using a function depending on both $\sigma_m^2$ and $R_v$
\begin{eqnarray}\label{eq:Bernstein-Sterfling}
\mathbb{P}\Big(\Big\lvert\Big\lvert\frac{1}{m}\sum_{i=1}^{m}{V_i} - \mu\Big\lvert\Big\lvert\leq \epsilon\Big)~.
\end{eqnarray}
We call it \textit{Sterfling} because the quality of the bound will depend on the sampling ratio. \cite{Schn16} shows an Hoeffding-Sterfling bound (i.e. not depending on $\sigma^2$) on $(2,D)-$Banach spaces, while \citep{Bard15} provided a Bernstein-Sterfling bound for real-valued random variable. Here we expand the result of \citep{Schn16} to the case of Bennet-Sterfling inequality in $(2,D)-$Banach spaces. We exploit the forward martingale \citep{Serf74,Bard15,Schn16} associated to the sampling without replacement and plug it into a sligthly modified result from \citep{pinelis1994optimum}.

For completeness, we recall the definition of $(2,D)-$ Banach spaces \citep[Definition 3]{Schn16} and refer to \citep[section 3]{Schn16} for more details. 
\begin{definition}\label{def:2D_smooth}
A  Banach space $(\mathcal{B},||\cdot||)$ is $(2,D)-$smooth if it a Banach space and there exists $D>0$ such that
\[
||\xx+\yy||^2+||\xx-\yy||^2\leq 2 ||\xx||^2 + 2r||\yy||^2~,
\]
for all $\xx,\yy\in\mathcal{B}$.
\end{definition}

Using Banach spaces allows to endow our space with non-Euclidean norms which can lead to important gains in measuring the variance.

\subsubsection{Forward Martingale when Sampling without Replacement}
Write $\bar{\vv} = \frac{1}{N}\sum_{i=1}^{N}{\vv_i}$ and consider $(M_k)_{k\in\mathbb{N}}$ the following random process
\begin{eqnarray}\label{eq:def_martingale}
M_k = \left\{
    \begin{array}{ll}
        \frac{1}{N-k}\sum_{i=1}^{k}{(V_i - \bar{\vv} )} & 1\leq k\leq m\\
        M_n & \text{for } k > m~.
    \end{array}
\right.
\end{eqnarray}
It is a standard result (when $m=N-1$) that $(M_k)_{k\in\mathbb{N}}$ defines a forward martingale \citep{Serf74,Bard15,Schn16} w.r.t. the filtration $(\mathcal{F}_k)_{k\in\mathbb{N}}$ defined as:
\begin{eqnarray}\label{eq:def_filtration}
\mathcal{F}_k = \left\{
    \begin{array}{ll}
        \sigma(V_1,\ldots,V_k) & 1\leq k\leq m\\
        \sigma(V_1,\ldots,V_n) & \text{for } k > m~.
    \end{array}
\right.
\end{eqnarray}

In fact the martingale defined in \eqref{eq:def_martingale} for some $m_0$ is also the stopped martingale at $m_0$ of the martingale in \eqref{eq:def_martingale} defined for $m=N-1$ (which corresponds to the martingale studied in \citep[Lemma 1]{Schn16}).

\begin{lemma}\label{lem:proof_martingale}
For $m\in[N-1]$, $(M_k)_{k\in\mathbb{N}}$ as defined in \eqref{eq:def_martingale} is a forward martingale with respect to the filtration $(\mathcal{F}_k)_{k\in\mathbb{N}}$ in \eqref{eq:def_filtration}.
\end{lemma}
\begin{proof}
For $1\leq k \leq m$, it is exactly the same computations as in \citep[Lemma 1.]{Schn16}. By definition, for $k > m$
\[
\mathbb{E}(M_k~|~\mathcal{F}_{k-1}) = \mathbb{E}(M_m~|~\mathcal{F}_m) = M_m = M_{k-1}~.
\]
\end{proof}

Importantly we also have the two following relations \citep[(3) and (5)]{Schn16}
\BEA
M_k-M_{k-1} = \frac{V_k-\mathbb{E}_{k-1}(V_k)}{N-k}\label{eq:rewrite}\\
||M_k-M_{k-1}||\leq \frac{R_v}{N-k}\label{eq:bound}~.
\EEA

\subsubsection{Bennet for Martingales in Smooth Banach Spaces}
We recall a sligthly modified version of \citep[Theorem 3.4.]{pinelis1994optimum}. This theorem is the analogous on martingales evolving on Banach spaces of Bennet concentration inequality for sums of real independent random variables.

\begin{theorem}[Pinelis]\label{th:pinelis_bennet}
Suppose $(M_k)_{k\in\mathbb{N}}$ is a martingale of a $(2,D)-$smooth separable Banach space and that there exists $(a,b)\in\mathbb{R}_{+}^*$ such that
\BEAS
\big\vert\big\vert\sup_k{||M_k-M_{k-1}||} \big\vert\big\vert_{\infty} &\leq& a\\
\big\vert\big\vert    \big( \sum_{j=1}^{\infty}{\mathbb{E}_{j-1}||M_j-M_{j-1}||^2} \big)^{1/2}   \big\vert\big\vert_{\infty} &\leq& b/D~,
\EEAS
then for all $\eta \geq 0$,
\[
\mathbb{P}(\sup_k{||M_k||}\geq \eta) \leq 2 \exp\big(-\frac{\eta^2}{2(b^2+\eta a/3)}\big)~.
\]
\end{theorem}

\begin{proof}
In the proof of \citep[theorem 3.4.]{pinelis1994optimum}, we have
\begin{eqnarray}\label{eq:pinelis_in_proof}
\mathbb{P}(\sup_k{||M_k||}\geq \eta) \leq 2\exp\big( -\lambda \eta + \frac{\exp(\lambda a)-1-\lambda a}{a^2}b^2\big)~.
\end{eqnarray}
Besides, from \citep[equation (16)]{sridharan2002gentle} we have
\[
\inf_{\lambda>0}\big[-\lambda\epsilon +(e^{-\lambda}-\lambda-1) c^2\big]\leq -\frac{\epsilon^2}{2(c^2+\epsilon/3)}~.
\]
We can rewrite \eqref{eq:pinelis_in_proof} as
\begin{eqnarray*}
\mathbb{P}(\sup_k{||M_k||}\geq \eta) &\leq& 2\exp\big( -\lambda a \frac{\eta}{a} + (\exp(\lambda a)-1-\lambda a )\frac{b^2}{a^2} \big)\\
&\leq& 2\exp\big( -\frac{\eta^2}{2({b^2}+{\eta a}/{3})}\big)~.
\end{eqnarray*}
\citep{pinelis1994optimum} uses the exact minimization on $\lambda$ which leads to a better but non standard form for the Bennet concentration inequality.
\end{proof}

\subsubsection{Bennet-Sterfling in Smooth Banach Spaces}

The following lemma allows to identify the constants $(a,b)$ appearing in theorem \ref{th:pinelis_bennet}.

\begin{lemma}\label{lemma:constant_pinelis}
\BEA
\big\vert\big\vert\sup_k{||M_k-M_{k-1}||} \big\vert\big\vert_{\infty} &\leq& \frac{R_v}{N-m}\label{eq:bound_a}\\
\big\vert\big\vert    \big( \sum_{j=1}^{\infty}{\mathbb{E}_{j-1}||M_j-M_{j-1}||^2} \big)^{1/2}   \big\vert\big\vert_{\infty} &\leq&  \sigma_m \frac{\sqrt{m}}{\sqrt{(N-m-1)N}} \nonumber ~,
\EEA
with $\sigma_m$ as in \eqref{def:variance_without_replacement}.
\end{lemma}

\begin{proof}
\eqref{eq:bound_a} directly follows from \eqref{eq:bound}.
Because of \eqref{eq:rewrite}, we have 
\BEAS
\sum_{k=1}^{\infty}{\mathbb{E}_{k-1}(||M_k-M_{k-1}||^2)} & = & \sum_{k=1}^{m}{\frac{1}{(N-k)^2}\mathbb{E}_{k-1}(||V_k-\mathbb{E}_{k-1}(V_k)||^2)}~.
\EEAS
Because of \eqref{def:variance_without_replacement}, we have,
\begin{eqnarray*}
\sum_{k=1}^{\infty}{\mathbb{E}_{k-1}(||M_k-M_{k-1}||^2)} & = & \sigma^2 \sum_{k=1}^{m}{\frac{1}{(N-k)^2}}~.
\end{eqnarray*}
Because of Lemma 2.1. in \citep{Serf74}, we have
\begin{eqnarray*}
\sum_{k=1}^{m}{\frac{1}{(N-k)^2}} &=& \sum_{k=N-m-1+1}^{N-1}{\frac{1}{k^2}}\\
 &\leq& \frac{m}{N(N-m-1)}~.
\end{eqnarray*}
It leads to 
\begin{eqnarray*}
\sum_{k=1}^{\infty}{\mathbb{E}_{k-1}(||M_k-M_{k-1}||^2)} & \leq & \sigma^2 \frac{m}{N(N-m-1)}
\end{eqnarray*}
and the desired result.
\end{proof}

\begin{theorem}\label{th:Bennet_Sterfling}
Consider $V$ a discrete set of $N$ vectors in a $(2,D)-$Banach space and $(V_i)_{i=1,\ldots,m}$ the random variables obtained by sampling without replacements $m$ elements of $V$. For any $\epsilon>0$,
\BEAS
\mathbb{P}\Big(\Big\lvert\Big\lvert\frac{1}{m}\sum_{i=1}^{m}{V_i} - \bar{\vv} \Big\lvert\Big\lvert \geq \epsilon \Big) & \leq & 2\exp\Big( - \frac{m\epsilon^2}{2\big(2 D^2 \frac{N-m}{N}\sigma^2 +\epsilon R_v/3\big)}\Big)~,
\EEAS
with $R_v\triangleq \sup_{\vv\in V}||\vv||$, and
\[
\sigma_m \triangleq \frac{1}{\sqrt{\sum_{k=1}^{m}{\frac{1}{(N-k)^2}}}} \Big\vert\Big\vert\Big(\sum_{k=1}^{m}{\frac{1}{(N-k)^2}\mathbb{E}_{k-1} \big\vert\big\vert V_k - \mathbb{E}_{k-1}V_k\big\vert\big\vert^2}\Big)^{1/2} \Big\vert\Big\vert_{\infty}~.
\]
\end{theorem}

\begin{proof}
Using Theorem \ref{th:pinelis_bennet} with the forward martingale \eqref{eq:def_martingale}, we have for any $\eta>0$,
\BEAS
\mathbb{P}\Big(\frac{1}{N-m}\Big\lvert\Big\lvert\sum_{i=1}^{m}{(V_i - \bar{\vv} )}\Big\lvert\Big\lvert \geq \eta\Big) &\leq & \mathbb{P}(\sup_i||M_i||\geq D) \nonumber\\
\mathbb{P}\Big(\Big\lvert\Big\lvert\frac{1}{m}\sum_{i=1}^{m}{V_i} - \bar{\vv} \Big\lvert\Big\lvert \geq \frac{N-m}{m}\eta\Big) & \leq & 2 \exp\big(-\frac{\eta^2}{2(b^2+\eta a/3)}\big)~.
\EEAS
Because of lemma \ref{lemma:constant_pinelis}, $a = \frac{R_v}{N-m}$ and $b = D \sigma_m \frac{\sqrt{n}}{\sqrt{N(N-m-1)}}$ is a good choice and leads to
\BEAS
\mathbb{P}\Big(\Big\lvert\Big\lvert\frac{1}{m}\sum_{i=1}^{m}{V_i} - \bar{\vv} \Big\lvert\Big\lvert \geq \frac{N-m}{m}\eta\Big) & \leq & 2\exp\big( - \frac{m}{(N-m)^2} \frac{m\epsilon}{2\big(D^2 \frac{m}{N(N-m-1)}\sigma_m^2 +\frac{m}{(N-m)^2}\epsilon R_v/3\big)}\big)\\
& \leq & 2\exp\big( - \frac{m\epsilon}{2\big(2 D^2 \frac{N-m}{N}\sigma_m^2 +\epsilon R_v/3\big)}\big)~,
\EEAS
for any $\eta >0$ with $\epsilon = \frac{N-m}{m}\eta$.
\end{proof}

\subsubsection{Approximate Caratheodory with High Sampling Ratio and Low Variance}

The primary tool for proving Approximate Caratheodory is to find a lower bound on the sampling ratio sufficient for the tail of the distribution at given level $\epsilon_0$ not to exceed a given probability $\delta_0$. With the Bennet-Sterfling inequality, we express a lower bound in the following lemma.

\begin{lemma}\label{lemma:inverse_proba_inequality}
In the setting of Theorem~\ref{th:Bennet_Sterfling}, for any $\delta_0\in]0,1[$ and $\epsilon_0>0$, if the sampling ratio $\alpha_m$ satisfies
\BEAS
\alpha_m &\geq & \frac{2 \ln(2/\delta_0) \big[ 2(D\sigma)^2 +\epsilon_0 R_v/3\big]/N}{\epsilon_0^2 + 2 \ln(2/\delta_0) \big[2(D\sigma_m)^2\big]/N}~,
\EEAS
we have
\BEQ\label{eq:goal_inverse_probability}
\mathbb{P}\Big(\Big\lvert\Big\lvert\frac{1}{m}\sum_{i=1}^{m}{V_i} - \bar{\vv} \Big\lvert\Big\lvert \geq \epsilon_0\Big)  \leq  \delta_0~.
\EEQ
\end{lemma}

\begin{proof}
Given $\delta_0\in]0,1[$ and $\epsilon_0>0$, we are looking for a sampling ratio $\alpha_m = \frac{m}{N}$ such that
\[
\mathbb{P}\Big(\Big\lvert\Big\lvert\frac{1}{m}\sum_{i=1}^{m}{V_i} - \bar{\vv} \Big\lvert\Big\lvert \geq \epsilon_0\Big) \leq  \delta_0~.
\]
With Bennet-Sterfling concentration inequality, it is sufficient to find $\alpha_m$ such that
\BEAS
2\exp\big( - \frac{m\epsilon^2}{2\big(2 D^2 \frac{N-m}{N}\sigma_m^2 +\epsilon R_v/3\big)}\big) &\leq & \delta_0\\
-\frac{N\alpha_m\epsilon^2}{2(D\sigma_m)^2(1-\alpha_m) + \epsilon R_v/3} &\leq& 2\ln(\delta_0/2)\\
\alpha_m\epsilon^2 &\geq & -\frac{2}{N} \ln(\delta_0/2)\big[2(D\sigma_m)^2(1-\alpha_m) +\epsilon R_v/3\big]\\
\alpha_m \big[\epsilon^2 - \frac{2}{N} 2(D\sigma_m)^2\ln(\delta_0/2)\big] &\geq & -\frac{2}{N}\ln(\delta_0/2)\big[ 2(D\sigma_m)^2 +\epsilon R_v/3\big]\\
\alpha_m &\geq & - \frac{\frac{2}{N}\ln(\delta_0/2)\big[ 2(D\sigma_m)^2 +\epsilon R_v/3\big]}{\epsilon^2 - \frac{2}{N}\ln(\delta_0/2) 2(D\sigma_m)^2}~.
\EEAS
For \eqref{eq:goal_inverse_probability} to be true, it is sufficient that $\alpha_m$ satisfies the following,
\[
\alpha_m \geq  \frac{2 \ln(2/\delta_0) \big[ 2(D\sigma_m)^2 +\epsilon_0 R_v/3\big]/N}{\epsilon_0^2 + 2 \ln(2/\delta_0) \big[2(D\sigma_m)^2\big]/N}~.
\]
which is the desired result.
\end{proof}

Using the normalization of Theorem \ref{th:approx-cara-high}, we get
\[
\alpha_m \geq  \frac{2 \ln(2/\delta_0) \big[ 2(D\sigma_m)^2 +\epsilon_0 R_v/(3N)\big]N}{\epsilon_0^2 + 2 \ln(2/\delta_0) \big[2(D\sigma_m)^2\big]N}~.
\]
and the leading term is controlled by the variance.
\end{document}